\documentclass[12pt]{article}
\title{Multicolour Poisson Matching}
\author{Gideon Amir, Omer Angel, and Alexander E.~Holroyd}
\date{20 May 2016}

\usepackage[colorlinks=true,linkcolor=black,citecolor=black,urlcolor=blue,pdfborder={0 0 0}]{hyperref}

\usepackage{amsfonts,amsmath,amsthm,amssymb}
\usepackage{graphicx,color}
\usepackage[protrusion=true,expansion=true]{microtype}
\usepackage{enumitem}
\usepackage{multirow}
\usepackage{tabu}

\usepackage[left=35mm,right=35mm]{geometry}

\usepackage{graphicx}
\usepackage[font=it, labelfont={bf}, margin=1cm]{caption}

\usepackage[nameinlink]{cleveref}
\crefname{thm}{Theorem}{Theorems}
\crefname{prop}{Proposition}{Propositions}
\crefname{lemma}{Lemma}{Lemmas}
\crefname{coro}{Corollary}{Corollaries}
\crefname{section}{Section}{Sections}
\crefname{figure}{Figure}{Figures}
\usepackage{enumitem}

\newcommand{\draftnote}[2]{{\bf [[\textcolor{red}{#1}:\textcolor{blue}{#2}]]}}

\newcommand{\note}[1]{\draftnote{}{#1}}

    \newtheorem{thm}{Theorem}
    \newtheorem{lemma}[thm]{Lemma}
    \newtheorem{prop}[thm]{Proposition}
    \newtheorem{coro}[thm]{Corollary}

\theoremstyle{definition} 

    \newtheorem*{remark*}{Remark}

\theoremstyle{remark} 

\renewcommand{\P}{\mathbb P}
\newcommand{\R}{\mathbb R}
\newcommand{\Z}{\mathbb Z}
\newcommand{\N}{\mathbb{N}}
\newcommand{\E}{\mathbb{E}}
\newcommand{\M}{{\cal M}}
\newcommand{\bv}{{v}}

\newcommand{\df}{\bf\boldmath}
\newcommand{\eps}{\varepsilon}
\newcommand{\dseg}[1]{\overrightarrow{#1}}
\newcommand{\lattice}{\mathcal L}
\newcommand{\eqd}{\stackrel{d}{=}}

\DeclareMathOperator{\diam}{diam}
\DeclareMathOperator{\dist}{dist}
\DeclareMathOperator{\geom}{Geom}
\DeclareMathOperator{\Exp}{Exp}
\DeclareMathOperator{\Bin}{Bin}
\DeclareMathOperator{\cone}{cone}

\newcounter{mycount}

\begin{document}

\maketitle

\begin{abstract}
  Consider several independent Poisson point processes on $\R^d$, each with
  a different colour and perhaps a different intensity, and suppose we are
  given a set of allowed family types, each of which is a multiset of
  colours such as red-blue or red-red-green.  We study
  translation-invariant schemes for partitioning the points into families
  of allowed types.  This generalizes the 1-colour and 2-colour matching
  schemes studied previously (where the sets of allowed family types are
  the singletons \{red-red\} and \{red-blue\} respectively).  We
  characterize when such a scheme exists, as well as the optimal tail
  behaviour of a typical family diameter. The latter has two different
  regimes that are analogous to the 1-colour and 2-colour cases, and
  correspond to the intensity vector lying in the interior and boundary of
  the existence region respectively.

  We also address the effect of requiring the partition to be a
  deterministic function (i.e.\ a factor) of the points. Here we find the
  optimal tail behaviour in dimension $1$.  There is a further separation
  into two regimes, governed by algebraic properties of the allowed family
  types.
\end{abstract}

\renewcommand{\thefootnote}{}
\footnotetext{{\bf\hspace{-6mm}Key words:} Poisson process,
point process, invariant matching, invariant partition,
factor map.} \footnotetext{{\bf\hspace{-6mm}AMS 2010
Mathematics Subject Classifications:} 60D05; 60G55; 05C70}

\section{Introduction}

The following random matching model was studied by Holroyd,
Pemantle, Peres and Schramm \cite{HPPS}.  Given two
independent homogeneous Poisson processes (called red and
blue) in $\R^d$, and a translation-invariant scheme for
bijectively matching red to blue points, what tail
behaviour is possible for the distance $X$ from a typical
point to its partner in the matching?  It turns out that
the answer is highly dependent on dimension.  For $d\geq 3$
there exist matching schemes in which $X^d$ has an
exponential tail, while for $d=1,2$, every matching scheme
has $\E X^{d/2} = \infty$. These bounds are essentially
optimal.  On the other hand, one may consider a Poisson
process of a single colour, and ask for a matching that
partitions the points into pairs.  In this case, there
exist matching schemes where $X^d$ has exponential tails in
all dimensions.  See \cite{HPPS} for proofs of these facts
and various related results.

In this article we consider extensions to arbitrary matching rules between
Poisson points of multiple colours.  For example, suppose that the red and
blue processes have different intensities, and that blue points must be
matched to red points, but red points are allowed to match to points of
either colour.  What is the best tail behaviour of the matching distance
$X$ that can be achieved?  Alternatively, suppose that points have three
colours (red, blue and green), and must be matched in pairs that contain
points of two distinct colours.  Or, suppose that the points must be
arranged into triplets consisting of a point of each colour.  We analyse a
general case that includes all the above examples.  It turns out that there
are three possibilities: either no translation-invariant matching exists,
or the optimal tail behaviour is similar to that for two-colour matching,
or to that for one-colour matching (as discussed above).  We give a
criterion for determining which case holds in terms of the matching rule
and the intensities of the processes of each colour.

To describe the general case we introduce some notation. Let $S_1,\ldots,S_q$
be disjoint sets (of points) with union $S$. We say that elements of $S_i$
have {\df colour} $i$.  Let $\N=\{0,1,\dots\}$. The {\df type} of a finite
set $F\subset S$ is the vector $(\#(F\cap S_i))_{i=1}^q \in \N^q$ specifying
the number of points of each colour. Let $V=\{\bv^1,\dots,\bv^k\}\subset
\N^q$ be a finite set of allowed types.  A \mbox{{\df
  $V$-matching}} of $(S_1,\dots,S_q)$ is a partition of $S$ into finite
sets, called {\df families}, each of which has type lying in $V$. For
example, if $q=2$ and $V=\{(1,1)\}$, a $V$-matching of $(S_1,S_2)$ is just
a perfect matching of the points of $S_1$ with the points of $S_2$
(equivalently, a bijection).

The support of a simple point process $\Pi$ is denoted by
\[
[\Pi]:=\{x:\Pi(\{x\})=1\}.
\]
Let $\Pi_1,\dots,\Pi_q$ be disjointly supported simple point processes on
$\R^d$. We sometimes consider the vector-valued process $\Pi$ given by
$\Pi(\cdot)=(\Pi_1(\cdot),\ldots,\Pi_q(\cdot))$, and call elements of
$[\Pi_i]$ {\df points} of $\Pi$ of {\df colour} $i$.  A {\df $V$-matching
  scheme} for $\Pi_1,\dots,\Pi_q$ is a simple point process $\M$ on
unordered finite subsets of $\R^d$ such that almost surely $[\M]$ is a
$V$-matching of $([\Pi_1],\dots,[\Pi_q])$. We say that $\M$ is {\df
translation-invariant} if the joint law of $(\M,\Pi_1,\dots,\Pi_q)$ is
invariant under the (diagonal) action of translations of $\R^d$.  Note that
(for the time being) $\M$ is not required to be a function of
$(\Pi_1,\dots,\Pi_q)$.

Let $\M$ be a translation-invariant $V$-matching scheme, and write
$\Psi=\sum_i\Pi_i$.  For a point $x\in[\Psi]$ we write $\M(x)$ for the unique
family that contains $x$.  For a set $S\subset \R^d$ write $\diam(S)$ for its
(Euclidean) diameter.  We are primarily interested in $\diam(\M(x))$ for a
``typical'' point $x\in[\Psi]$.  To make this precise, define
\[
F(r):=\frac{1}{\E\Psi(D)} \E \#\Big\{x\in [\Psi]\cap D: \diam
[{\cal M}(x)] \leq r\Big\},
\]
where $D$ is some set with positive finite Lebesgue measure. (In the translation
invariant cases we consider, $F$ is independent of the choice of $D$.)
Note that $F$ is a distribution function.  We introduce a random
variable $X$ with law $\P^*$ and expectation operator $\E^*$ such that
\[
\P^*(X \leq r) = F(r) \qquad \forall r.
\]
The random variable $X$ represents the diameter of the family of a typical
point. We call $X$ the \textbf{typical diameter} of $\M$.
The random variable $X$ may be interpreted as $\diam(\M(0))$ under a Palm
process derived from $\M$ (see e.g.\ \cite[Chapter 11]{kall}).

Our first main result is a trichotomy for the law of $X$.
For a set $A\subset\R^q$, we denote its boundary (resp.\
interior) by $\partial A$ (resp.\ $A^\circ$). Let
$\cone(V)$ denote the cone spanned by the allowed family
types $V = \{\bv^1,\dots,\bv^k\} \subset\N^q$, defined by
\[
\cone(V) := \Big\{\sum_{i=1}^k a^i \bv^i: a^1,\dots,a^k\in [0,\infty)
\Big\} \subset \R^q.
\]

\enlargethispage*{2cm}
\begin{thm}\label{T:main}
  Let $(\Pi_i)_{i=1,\dots,q}$ be independent homogeneous Poisson point
  processes on $\R^d$ with respective intensities $\lambda_i\in(0,\infty)$.
  Let $V\subset \N^q$ be a finite set
  not containing every unit vector of $\N^q$.
  \begin{enumerate}[label={\textup{(\roman*)}},nosep,leftmargin=18mm,itemindent=-8mm]
  \item If $\lambda \notin \cone(V)$:  \label{c:unsatisfiable}

    no translation-invariant $V$-matching scheme exists.

  \item If $\lambda\in\partial(\cone(V))$ and $d\leq 2$: \label{critical}

    there exists a translation-invariant $V$-matching scheme such that
    $\P^*(X>r)\leq C r^{-d/2}\;\forall r$, while every
    translation-invariant $V$-matching scheme satisfies $\E^* X^{d/2} =
    \infty$.

  \item If either $\lambda\in (\cone(V))^\circ$, or
    $\lambda\in\partial\cone(V)$ and $d\geq 3$: \label{c:underconstrained}

    there exists a translation-invariant $V$-matching scheme such that
    $\P^*(X>r) \leq e^{-C r^d}\;\forall r$, while every
    translation-invariant $V$-matching scheme satisfies $\P^*(X>r)\geq
    e^{-c r^d}\;\forall r$.
  \end{enumerate}
  Throughout, $c,C$ are positive finite constants depending on $d$,
  $\lambda$ and $V$ but not $r$.
\end{thm}

Note that since $V$ is finite, $\cone(V)$ is a closed set,
and so the three cases are mutually exclusive and cover all
possible $\lambda$.  If all unit vectors are in $V$ then
the trivial matching with all singletons has $X=0$ a.s.,
which is of no interest.  The case $\lambda \notin
\cone(V)$ is referred to as {\bf unsatisfiable}.  The case
$\lambda\in\partial\cone(V)$ is {\bf critical}. The case
$\lambda\in\cone(V)^\circ$ is {\bf underconstrained}. Note
that (with respect to the tail of $X$) the critical case
behaves like the underconstrained case in dimensions $d>2$.

Here are several examples of special cases of \cref{T:main}, starting with the
two cases considered in \cite{HPPS}.

\begin{enumerate}
\item {\bf 1-colour matching}. Let $\lambda=(1)$ and $V=\{(2)\}$. (All
  points are the same colour, and each family must contain two points). This
  is a underconstrained setting.  Indeed, every matching problem with a single
  colour is underconstrained.

\item {\bf 2-colour matching}. Let $\lambda=(\lambda_1,\lambda_2)$ and
    $V=\{(1,1)\}$. (Each family comprises a red and a blue point.) This
    case is critical if $\lambda_1=\lambda_2$, and otherwise
    unsatisfiable.
\end{enumerate}

\noindent
(The above two cases of \cref{T:main} were proved in \cite{HPPS}.)

\begin{enumerate}[resume]
\item {\bf Bisexuality}. Let $\lambda=(\lambda_1,\lambda_2)$ and
  $V=\{(2,0),(1,1)\}$. (Each red point must be matched to a blue point, but
  a blue point may be matched to another point of either colour).  This is
  unsatisfiable if $\lambda_1<\lambda_2$, critical if
  $\lambda_1=\lambda_2$, and underconstrained if $\lambda_1>\lambda_2$.

\item {\bf Triplets}. Let $\lambda=(1,1,1)$ and $V=\{(1,1,1)\}$. (Red, blue
  and green points have equal intensities, and a family must contain of one
  of each colour). This setting is critical.

\item {\bf Single family type}. Generalizing the previous examples, suppose
  $V=\{v\}$ consists of a single family type $v\in\N^q$.  If there is a
  single colour this is underconstrained.  If there is more than one colour
  and $\lambda=av$ for some $a$ this is critical, while if $\lambda$ is not
  a multiple of $v$ this is unsatisfiable.

\item \label{eg:triangle} {\bf Colourful matching}. Let $V = \{(1,1,0),
    (1,0,1), (0,1,1)\}$. (Red, green and blue points must be matched into
    pairs containing distinct colours.) If all colours have the same
    intensity, $\lambda=(1,1,1)$, then this is underconstrained.  Moreover,
    the same holds as long as the entries of $\lambda$ form a
    non-degenerate triangle.  If the triangle inequality is violated this
    setting becomes unsatisfiable, while a degenerate triangle (where one
    intensity equals the sum of the others) is critical.
\end{enumerate}

We also consider the question of whether it is possible to have a
\textbf{factor} matching, i.e.\ a matching that is a deterministic function
of the Poisson processes $(\Pi_1,\ldots,\Pi_q)$, and, if so, what can be said
about the tail of $X$ for factor matchings.  In the one dimensional case, we
answer this in the following theorem.  A central player here is the lattice
spanned by the allowed family types.  For allowed family types
$\{\bv^1,\dots,\bv^k\}$, define the lattice
\[
\lattice = \lattice(V) := \Bigl\{\sum_i n^i \bv^i : \ n^1,\dots,n^k \in
  \Z\Bigr\}\subset \Z^q.
\]

\begin{samepage}
\begin{thm}[Factor matchings]\label{T:factor}
  Consider dimension $d=1$.  Let $(\Pi_i)_{i=1,\dots,q}$ be independent
  homogeneous Poisson point processes on $\R$ with respective intensities
  $\lambda_i\in(0,\infty)$.  Let $V\subset \N^q$ be a finite set not
  containing every unit vector of $\N^q$.
  \begin{enumerate}[label={\textup{(\roman*)}},nosep,leftmargin=18mm,itemindent=-8mm]
  \item If $\lambda\in\partial\cone(V)$: \label{factor_crit}

    there exists a translation invariant matching that is a factor of the
    Poisson processes with $\P^*(X>t) \leq C/\surd{t}$ for some constant
    $C$.

  \item If $\lambda\in\cone(V)^\circ$ and $\lattice=\Z^q$: \label{factor_exp}

    there exists a translation invariant matching that is a factor of the
    Poisson processes with $\P^*(X>t) \leq Ce^{-ct}$ for some constants
    $c,C$.

  \item If $\lambda\in\cone(V)^\circ$ and $\lattice\neq\Z^q$:
    \label{factor_cauchy}

    there exists a translation invariant matching that is a factor of the
    Poisson processes with $\P^*(X>t) \leq C/t$ for some constant $C$,
    and any translation invariant matching that is a factor of the
    Poisson processes has $\E^* X = \infty$.
  \end{enumerate}
\end{thm}
\end{samepage}

Note that \cref{T:main} (ii) and (iii) give complementary lower bounds to
\cref{T:factor} (i) and (ii): $\E^*\surd X=\infty$ and $\P^*(X>r)\geq
 e^{-c r}$ respectively.  \cref{T:main} (i) covers the
case $\lambda\notin\cone(V)$.  In the case $\lambda\in\cone(V)^\circ$, this
theorem shows that when $\lattice\neq\Z^q$ the possible tail behaviours of
$X$ change significantly when we restrict to factor matchings.

The bound $\E^* X = \infty$ in \cref{T:factor}(iii) is an extension of a
parity argument from \cite{HPPS}, and is specific to the 1-dimensional
case.  The constructions of matchings for all parts of this theorem are
much more intricate.
We believe that the dichotomy according to whether $\lattice=\Z^q$ or not is
peculiar to dimension one, so that in higher dimensions, the claims of
\cref{T:main} about the tail of $X$ hold for factor matchings as well
(perhaps with different constants). In particular, we expect that in the
underconstrained case, and also for $d>2$ in the critical case, there are factor
matchings with $\P^*(X>t)\leq Ce^{-ct^d}$, even if $\lattice\neq\Z^q$.  See
\cite{HPPS,Timar} for further results on existence and properties of factor
matchings.

\pagebreak
Here are some further examples.

\begin{enumerate}[resume]
\item {\bf Single colour.}  For a single Poisson process on $\R$, if there
  are families of only one size $a$ then $\lattice=a\Z$, and any factor
  matching (in $d=1$) has $\E^* X=\infty$.  However, if allowed family
  sizes have greatest common divisor $1$ (for example, if $V=\{(2),(3)\}$,
  so points can be matched in twos or threes) there is a factor matching
  with exponential tail.

\item {\bf Partial two-colour matching.}  Let
  $\lambda = (\lambda_1,\lambda_2)$ and $V=\{(1,1),(1,0)\}$, so a blue
  point must match to a red point, but a red point may also form a family
  on its own.  Again, $\lattice=\Z^2$.  When $d=1$, if
  $\lambda_1=\lambda_2$ then the bound $\P(X>t)<C/\surd t$ can be attained
  by a factor matching, while if $\lambda_1>\lambda_2$ then there is a
  factor matching with exponential tails.

\item {\bf Matching in pairs.}  In any setting where points are matched in
  pairs, with some constraints on which colour pairs are valid, the lattice
  is contained in the even lattice, and so is not $\Z^q$.  Thus any factor
  matching in one dimension has $\E^* X=\infty$. 
\end{enumerate}

Table~\ref{summary} summarizes the main results stated above.

\begin{table}[t]
\centering
\begin{tabu}{|c|c|[1.5pt]c|c|}
  \cline{3-4}
\multicolumn{2}{l}{} & \multicolumn{1}{|c|}{general} & factors ($d=1$)  \\
\tabucline[1.5pt]{3-4}
\cline{1-2}
\multicolumn{2}{|l|[1.5pt]}{$\lambda\notin\cone(V)$} & \multicolumn{2}{c|}{impossible [L1]} \\
\hline
\multicolumn{2}{|l|[1.5pt]}{  \multirow{2}{*}{$\lambda\in\partial\cone(V)$}  } & $d/2$ (if $d\leq 2$) \hfill [L1,U3] & \multirow{2}{*}{$1/2$\hspace{.2in} [L1,U3]} \\
 \multicolumn{2}{|l|[1.5pt]}{} & Exp (if $d\geq 3$) \hfill [U3] & \\
  \hline
\multirow{2}{*}{$\lambda\in \cone(V)^\circ$} & $\lattice\neq \Z^q$ & \multirow{2}{*}{Exp \hspace{1in} [U4]} & 1 \hfill [L2,U6] \\
\cline{2-2}\cline{4-4}
 & $\lattice = \Z^q$ & & Exp \hfill [U5] \\
\hline
\end{tabu}
\caption{A summary of the results of \cref{T:main,T:factor}, for general
$V$-matchings and factor $V$-matchings. An real number $\alpha$ indicates
that there exists a $V$-matching in which the typical family diameter $X$ has
all moments below the $\alpha$th moment finite, but none with finite
$\alpha$th moment. ``Exp" indicates that there exists a matching in which
$X^d$ has a finite exponential moment.  The arguments for the various upper
bounds (i.e.\ constructions) and lower bounds are as follows.  The lower
bounds all use extensions of arguments in \cite{HPPS}: those marked with [L1]
use mass-transport together with a ``charge function'' on colours, while [L2]
uses a modularity argument. The upper bounds [U3] use a reduction to
$2$-colour matchings, together with a Hilbert curve construction for $d\geq
2$.  The upper bound [U4] uses a natural greedy Markov matching procedure,
and again the Hilbert curve.  The constructions for [U5] and [U6] are the
most elaborate and novel: [U5] modifies the greedy construction using
randomness extracted from the point locations; and [U6] combines this method
with a multi-scale construction.
 } \label{summary}
\end{table}

\paragraph{Infinitely many types.} We now consider how the situation changes when
there are infinitely many valid family types.  This case is slightly more
delicate, particularly in the critical case. We still assume that the number
of colours is finite, since otherwise very little can be said (see the remark
below). As before, we divide our analysis into cases according to the
relation between the intensity vector $\lambda$ and $\cone(V)$.

\begin{thm}[Infinite $V$]\label{T:infinite}
  Let $(\Pi_i)_{i=1,\dots,q}$ be independent homogeneous Poisson point
  processes on $\R^d$ with respective intensities $\lambda_i\in(0,\infty)$.
  Let $V\subset \N^q$ be a (possibly infinite) set not containing every unit
  vector.  Then the clauses (i)--(iii) of \cref{T:main} hold, except that in clause
  (ii) the condition
  $\lambda\in \partial\cone(V)$ must be replaced with
  $\lambda\in\partial\cone(V) \cap \cone(V)$.
\end{thm}

To clarify the difference between this and \cref{T:main}, note that when
there are infinitely many family types, it is possible that $\cone(V)$ is
not closed.  For example, with family types $V=\{(n,n+1):n\in\N\}$, the
cone is $\{0\leq x < y\}$.  Thus it is possible that
$\lambda\notin\cone(V)$ but $\lambda\in\partial\cone(V)$.  For example
with that $V$, if the two intensities are equal there is no matching.
Increasing $\lambda_2$ by an arbitrarily small amount makes
matchings possible (and indeed, the setting becomes underconstrained).

\begin{remark*}
  With infinitely many colours fairly general tail behaviours can be
  forced. For instance, for any sequence of distances $r_k$ and sequence of
  probabilities $p_k$ it is not hard to construct a set of
  intensities and a countable family of types $V$ such that any
  translation invariant $V$-matching scheme satisfies $\P^*(X>r_k) > p_k$
  (e.g.\ by having colours with very low intensity that only take part in
  very large families).
\end{remark*}

\paragraph{Matching in pairs.}
Finally, we consider the natural special case when the matching
consists only of pairs of points, with some restrictions on which colour
pairs are allowed.  In the general formulation used above,
this corresponds to having $\|v\|_1=2$ for all $v\in V$. Such a
setting can be described in terms of a graph,
possibly with self-loops.  The vertices are the colours and an edge
indicates that two points of the corresponding colours can form a pair in
the matching.  Vertices $i$ and $j$ are neighbours in the graph if and only
if $e_i + e_j\in V$ (i.e., matching points of colours $i$ and $j$ is
allowed), and then we write $i\sim j$.  Here $e_i$ is the $i$th unit vector.

In this case, we can give alternative criteria for criticality and
unsatisfiability, similar to the conditions of the K\"onig-Hall marriage
theorem.  For a set $S\subset[q]$ define $N(S)$ to be the set of its
neighbours in the graph:
\[
N(S) := \{x : \exists y\in S \text{ such that } x\sim y\}.
\]
For a set $S$ we define $\lambda(S) = \sum_{i\in S} \lambda_i$ to be the
total intensity of points with colours in $S$.  Given the intensities
$\lambda$ and the graph, a non-empty set $S\subset[q]$ is called:
\begin{itemize}[nosep]
\item {\bf deficient} if $\lambda(N(S)) < \lambda(S)$,
\item {\bf critical} if $\lambda(N(S)) = \lambda(S)$ and $S\neq N(S)$, and
\item {\bf excessive} if $\lambda(N(S)) > \lambda(S)$.
\end{itemize}

The following relates existence of deficient and critical sets to the
location of $\lambda$ w.r.t.\ $\cone(V)$.  The corresponding case of
\cref{T:main} then applies.

\begin{prop}[Matching in pairs]\label{P:pairs} 
  Fix the intensity vector $\lambda\in(0,\infty)^q$ and let $V$ and the
  graph be as above.
  \begin{enumerate}[label={\textup{(\roman*)}},nosep]
  \item If there exists a deficient set $S\subset[q]$, then
    $\lambda\notin\cone(V)$ (and there is no translation invariant
    $V$-matching scheme).
  \item If there is no deficient set, but there is a critical
    $S\subset[q]$, then $\lambda\in\partial\cone(V)$.
  \item If all non-empty subsets $S\subset[q]$ are excessive or have
    $N(S)=S$, then $\lambda\in(\cone(V))^\circ$.
  \end{enumerate}
\end{prop}

For instance, in Example \ref{eg:triangle} above, (three colours and the
constraint is that pairs are of distinct colours), let the intensities of
the point processes be $\lambda_1,\lambda_2,\lambda_3$, and assume without
loss of generality $\lambda_1\geq\lambda_2,\lambda_3$.  If
$\lambda_1 > \lambda_2+\lambda_3$, then $S = \{1\}$ is a deficient set.
If $\lambda_1 = \lambda_2+\lambda_3$, then $S = \{1\}$ is critical.

The main issue in this setting is existence of a perfect weighted fractional
matching in the graph where the total weight of edges at vertex $i$ is
$\lambda_i$, which holds if an only if if there are no deficient sets.
In the case $\lambda \equiv 1$
this is due to Tutte, see \cite{tutte,fgt}.

\subsection{Further notation}


Recall that the number of distinct colours is denoted by $q$; the number of
family types is $k$ and the allowed families are $\bv^1,\dots,\bv^k$. For two
vectors $x = (x_1,\dots,x_m)$ and $y = (y_1,\ldots,y_m)$ in $\R^m$ we
denote the inner product $x \cdot y := \sum_{i=1}^m x_i y_i$.  We sometimes
treat the
set $V$ as a $k\times q$ matrix with rows $\bv^i\in\N^q$ (in some arbitrary
but fixed order), allowing us to write $a V := \sum_{i=1}^k a^i \bv^i$ for
any vector $a \in \R^k$.

Recall that
\[
  \lattice = \lattice(V) := \{n V \ : \ n \in \Z^k\}
\]
is the lattice spanned by $V$, and define also the non-negative lattice
\[
\lattice_+ = \lattice_+(V) := \{n V \ : \ n \in \N^k\}.
\]
(With the convention that $\N=\{0,1,\dots\}$.)  Note that
$\lattice_+\subset \big(\lattice \cap \cone(V) \big) \subset \N^q$.
However, in general the former inclusion is strict; see \cref{fig:cone} for an example.  A vector $x$ is called {\df matchable} if $x\in\lattice_+$, since a set
containing $x_i$ points of colour $i$ can be partitioned into valid
families.

\begin{figure}[t]
  \centering
  \includegraphics[width=.6\textwidth]{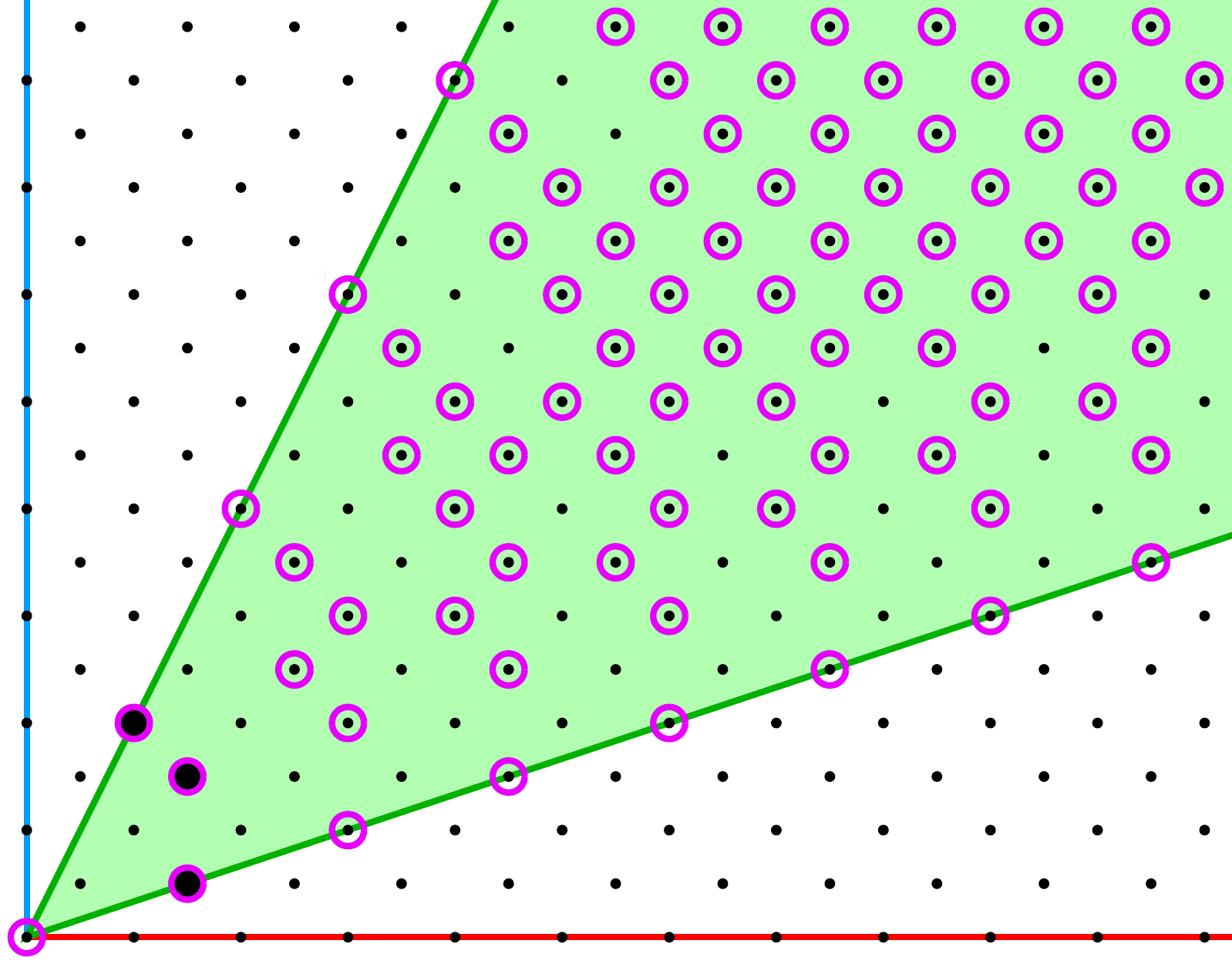}
  \caption{Key geometric objects for the matching rule
    $V=\{(3,1),(3,3),(2,4)\}$. (Families may consist of 1 red and 3 blue,
    or 3 red and 3 blue, or 4 red and 2 blue points.)  Elements of $V$ are
    marked with filled discs.  The set $\cone(V)$ is shaded: matchings with
    exponential tails are possible for intensity vectors
    $(\lambda_{\text{red}}, \lambda_{\text{blue}})$ in its interior, while
    for $d\le2$ only a power law is possible in its boundary.  The lattice
    $\lattice(V)$ is shown by the dots, and the matchable vectors
    $\lattice_+(V)$ (corresponding to sets of points that can be
    partitioned into families) are circled.  Since $\lattice\neq \Z^2$,
    factor matchings must have infinite mean family size in $d=1$, even in
    the interior of the cone.}
  \label{fig:cone}
\end{figure}

We denote by $C,c$ positive constants whose value may change from line to
line. Generally statements would hold for $c$ small enough and $C$ large
enough.

\subsection{Charge and mass transport}

As noted, the behaviours we get for matchings in general are similar to the
previously studied cases of one and two colour matchings.
A central new idea is to define a {\bf charge} function with useful properties.  We
will assign each colour $i$ a real number $\eta_i$ called the charge. We
think of charge $\eta_i$ as located on each point of colour $i$, and write
$\eta(x)=\eta_i$ for $x\in[\Pi_i]$.  We will choose $\eta$ so that the
total charge in each family is non-positive.  In the unsatisfiable case we
can do this in such a way that the average charge over space is positive,
which leads to a contradiction using the mass transport principle (see
below).  In the critical case the average charge is $0$, and conservation
is used to derive lower bounds on the tail of the matching distances. In
order to choose appropriate charges, we use hyperplane separation
 (see e.g.\ \cite[Chapter 11]{Rock}).

\begin{prop}[Hyperplane separation]\label{T:hyperplane}
  If $C,K\in\R^n$ are disjoint convex sets and $K$ is compact, then there
  exists a non-zero $\eta\in\R^n$ so that $\inf_{x\in K} (\eta\cdot x) \geq
  \sup_{y\in C} (\eta\cdot y)$.  If $C$ is closed then the inequality is
  strict.
\end{prop}

We use this for the singleton set $K=\{\lambda\}$, and $C=\cone(V)$.
Clearly for any $\eta$ and cone $C$ we have $\sup_{y\in C} (\eta\cdot y)
\in\{0,+\infty\}$; the inequality then implies that the supremum must be
$0$, so that the charge in each family is non-positive.

Another important tool is the mass-transport principle, which we use in the
measure-theoretic form below.  For background and extensions, see \cite{lyons,BS,BLPS}

\begin{lemma}[Mass transport]\label{L:CMTP}
  Let $\mu$ be a measure on $\R^d \times \R^d$ that is invariant under the
  diagonal action of translations, i.e.\ $\mu((A+x)\times(B+x)) =
  \mu(A\times B)$ for any $x\in\R^d$ and any Borel sets $A,B\subset\R^d$.
  Then $\mu(B\times\R^d) = \mu(\R^d\times B)$ for any Borel $B$.
\end{lemma}

In applications, $\mu$ is often taken to be the expectation of a diagonally
invariant \emph{random} measure, and then we think of $\mu(A\times B)$ as
the expected amount of mass sent from $A$ to $B$.  Then the mass transport
principle says that the total expected mass transported out of a set equals
the expected mass transported into it.

\begin{proof}
  Suppose first that $B$ is the unit cube $[0,1)^d$. Define a function on
  $\Z^d\times\Z^d$ by $f(x,y) = \mu((x+B)\times(y+B))$.  Invariance of
  $\mu$ implies that this function is invariant under the action of $\Z^d$,
  and so $f(0,x) = f(-x,0)$.  Summing this over $x\in\Z^d$ yields the claim
  for $B$ the unit cube (since all terms are non-negative, the order of
  summation can be changed.)

  Similarly, using $B=[t,t+a)^d$ and summing over $x\in(a\Z)^d$ we get the
  claim for cubes of side $a$. Unions give any open set $B$, and therefore
  also any Borel set $B$.
\end{proof}

\paragraph{Structure of the paper.}

\cref{sec.unsatisfiable,sec.underconstrained,sec.critical} contain proofs
involving the unsatisfiable, underconstrained and critical cases
respectively. \cref{T:factor} about factor matchings is proved in
\cref{s:factor}.  In \cref{sec.infinite} we prove \cref{T:infinite}
concerning the case of infinitely many allowed family types.
\cref{s:pairs} contains the proof of \cref{P:pairs} on colourful pair
matchings.  We end with some open questions in \cref{s:questions}.

\paragraph{Acknowledgment.}

This work was initiated during a UBC Probability Summer School, and
advanced while some of the authors were visiting Microsoft Research.  We
are grateful to Microsoft Research for their support.  OA is supported by
NSERC.

\section{The unsatisfiable case}
\label{sec.unsatisfiable}


\begin{proof}[Proof of \cref{T:main} (i)]
  As with most proofs based on mass transport, the key is to find a useful
  mass transport function. Given an invariant matching scheme, we show how
  to construct a mass transport that contradicts the principle. Since
  $\lambda\not\in\cone(V)$ and $\cone(V)$ is closed convex set, by
  \cref{T:hyperplane} there is a vector of charges $\eta\in\R^q$ such that
  $\eta\cdot\lambda > \sup\{\eta\cdot x : x\in\cone(V)\}$. Since $\cone(V)$
  is a cone, this supremum is in $\{0,+\infty\}$ and thus must be $0$, and
  so $\eta\cdot\lambda > 0$.  To apply the mass transport principle, it is
  convenient to work with non-negative charges.  To this end, we let
  $K = -\min_i \eta_i$. 
  By a slight abuse of notation, we let $\eta(x)=\eta_i$ for any point
  $x\in[\Pi_i]$.

  Suppose that $\M$ is a translation invariant $V$-matching, and recall
  that $\M(x)$ is the family of the matching that contains the point
  $x$. Define the translation invariant measure $\mu$ on $\R^d\times\R^d$
  by
  \[
  \mu(A\times B) = \E \sum_{x\in A\cap [\Psi]} \  \sum_{y\in\M(x)\cap B}
  \frac{K+\eta(x)}{\#\M(x)}.
  \]
  This corresponds to the mass transport in which each point $x\in[\Pi_i]$ sends
  out a total mass $K+\eta_i$ divided evenly to its family $\M(x)$, and no
  mass to points outside its family. The total mass received by a point $y$
  is $\sum_{x\in\M(y)} \frac{K+\eta(x)}{\#\M(y)} \leq K$, since the total
  $\eta$-charge in a family is non-positive.

  We apply \cref{L:CMTP} to $\mu$.  Let $B$ be a set of volume $1$. We have
  \begin{align*}
    \mu(\R^d\times B) &= \E \sum_{x\in[\Psi]} \ \sum_{y\in\M(x)\cap B}
    \frac{K+\eta(x)}{\#\M(x)} \\
    &= \E \sum_{y\in [\Psi] \cap B} \sum_{x\in\M(y)}
    \frac{K+\eta(x)}{\#\M(x)} \\
    &\leq \sum_i K \lambda_i ,
  \end{align*}
  since the inner sum on the second line is at most $K$ for any $y$.  However,
  \begin{align*}
    \mu(B\times\R^d) &= \E \sum_{x\in[\Psi]\cap B} K+\eta(x) \\
    &= \sum_i \lambda_i(K+\eta_i) > \sum_i K\lambda_i,
  \end{align*}
  since $\eta\cdot\lambda > 0$. The contradiction implies that an invariant
  $V$-matching does not exist.
\end{proof}

\section{The underconstrained case}
\label{sec.underconstrained}

While the cases of \cref{T:main} are split according to the tail behaviour
of $X$, the proofs are separate for the cases $\lambda\in\partial(\cone(V))$
and $\lambda\in\cone(V)^\circ$. We begin with the latter, forming part of
case (iii).

We first show the existence of an invariant matching scheme that gives the
desired tail bounds for the diameter of the family of a typical point. We
begin with the case of dimension $d=1$, and then use the one-dimensional case
to derive the claim for general $d$.

Assume $d=1$, and consider the process $\Pi(0,t] =
(\Pi_1(0,t], \dots, \Pi_q(0,t])$, taking values in $\N^q$ (for $t\geq0$).
Define also
\begin{equation}\label{eq:TM_def}
T = \inf\Big\{t>0 : \Pi(0,t]\in \lattice_+(V) \text{ and } \Pi(0,t]\neq 0\Big\},
\end{equation}
i.e.\ the first $t$ for which the vector $\Pi(0,t]$ is matchable
and non-zero.  The main step in the proof is the following lemma.

\begin{lemma}\label{L:TM_tail}
  Suppose $d=1$ and $\lambda\in\cone(V)^\circ$.  Then $T$ as defined in
  \eqref{eq:TM_def} has an exponential tail: there exist constants $C,c>0$
  such that $\P(T>t) < C e^{-c t}$ for any $t>0$.
\end{lemma}

The proof consists of three steps. We show that with high probability at
all large times, $\Pi(0,t]$ is ``well inside'' $\cone(V)$ in a certain
sense, that $\Pi(0,t]$ visits the lattice $\lattice$ regularly, and finally
that any point in $\lattice$ that is well inside the cone corresponds to a
matchable set.

We begin with a simple geometric statement.  Let $\|\cdot\|$ denote the
Euclidean norm, and $\dist(x,A)$ the Euclidean distance from the point $x$
to the set $A$.

\begin{lemma}\label{L:coneincone}
  Suppose $\lambda\in\cone(V)^\circ$. Let $\delta = \frac12\dist(\lambda,
  \cone(V)^\mathrm{c})$, and suppose $\pi\in\R^q$ satisfies $\|\pi - s\lambda\| <
  \delta s$ for some $s>0$, and let $\xi\in\R^q$. Then $s > \|\xi\|/\delta$
  implies $\pi \in \xi+\cone(V)$ (the translated cone).
\end{lemma}

\begin{proof}
  We have $\|(\pi-\xi)-s\lambda\| \leq \|\xi\| + \|\pi-s\lambda\| < 2\delta
  s$. By linearity, $2\delta s = \dist(s\lambda, \cone(V)^c)$, hence
  $\pi-\xi\in\cone(V)$.
\end{proof}


For $\alpha>0$ let
\[
\cone_\alpha(V):=\left\{\sum a^i \bv^i : a \in [\alpha,\infty)^k \right\}.
\]
Clearly $\cone_\alpha(V)\subset\cone(V)$ and is just a translation of the
cone, since $\cone_\alpha(V) = \cone(V) + \alpha \sum \bv^i$.

\begin{lemma}\label{L:in_cone_c}
  For any $V,\lambda$, satisfying $\lambda\in\cone(V)^\circ$, and any
  $\alpha>0$ there exist $C,c>0$ such that for any $t$,
  \begin{equation}
    \P\Bigl( \forall s>t : \Pi(0,s] \in \cone_\alpha(V) \Bigr)
    \geq 1 - C e^{-c t}
  \end{equation}
\end{lemma}

\begin{proof}
  Fix $\eps>0$. For any given $t$ we have
  \[
  \P\Big(|\Pi_i(0,t] - t\lambda_i| \geq \eps t\Big) \leq C e^{-c t},
  \]
  where $c,C$ depend only on $\eps,\lambda_i$.  A union bound shows that
  \[
  \P\Big(\forall \text{ integers } n>t : |\Pi_i(0,n] - n\lambda_i| < \eps n
  \Big) \geq 1-C e^{-c t},
  \]
  where only $C$ has changed. Since $\Pi_i(0,s]$ is monotone in $s$, as
  long as $t>\lambda_i/\eps$ it follows that
  \[
  \P\Big(\forall s>t : |\Pi_i(0,s] - s\lambda_i| < 2\eps s\Big) \geq 1-C
  e^{-c t},
  \]
  and by changing $C$ again, this holds for all $t$. Since this holds for
  each of $d$ coordinates, we get
  \begin{equation}\label{eq:Poisson_LD}
    \P\Big( \forall s>t : \|\Pi(0,s] - s\lambda\| < 2d\eps s\Big)
    > 1 - C e^{-c t}.
  \end{equation}

  Apply \eqref{eq:Poisson_LD} with $\eps = \frac1{4d} \dist( \lambda,
  \cone(V)^c)$, and let $\xi = \alpha\sum \bv^i$, so that $\cone_\alpha(V) =
  \xi + \cone(V)$. If $t>\|\xi\|/\delta$, then by \cref{L:coneincone} with
  exponentially high probability (in $t$), for all $s>t$ we have
  $\Pi(0,s]\in\cone_\alpha(V)$.

  This completes the proof for $t>\|\xi\|/\delta$. By adjusting $C$ we get
  the result for smaller $t$.
\end{proof}

\begin{lemma}\label{L:hit_lattice}
  Assume $\lambda\in\cone(V)^\circ$. Then for some constants and all $t$,
  \begin{equation}\label{eq:in_lambda}
    \P\Big(\exists s\in(t,2t) \text{ such that } \Pi(0,s] \in \lattice \Big)
    > 1 - C e^{-c t}.
  \end{equation}
\end{lemma}

\begin{proof}
  Since $\cone(V)$ has a non-empty interior, $V$ contains a basis for
  $\R^q$. (This is actually all we need to know about $\lambda$ and $V$ for
  this lemma.)
  Hence the lattice $\lattice$ has full dimension $q$, and so the quotient
  $\Z^q/\lattice$ is a finite group. Identifying $\Pi(0,t]$ with its coset
  $\Pi(0,t]+\lattice$, we see that the process $\{\Pi(0,t]\}_{t\geq0}$ is a
  continuous time random walk on a finite group. It is irreducible since
  the possible jumps include adding a single point of any colour, and so
  generate $\Z^q$ and its quotients. Thus the probability of avoiding the
  $0$ coset for time $t$ is exponentially small.
\end{proof}

Combining \cref{L:hit_lattice,L:in_cone_c} we have proved:

\begin{coro}\label{C:T_alpha}
  Let $T_\alpha := \min\{t\geq 0 : x(t)\in \cone_\alpha(V) \cap \lattice\}$,
  then there are $c,C$ depending only on $V$, $\lambda$ and $\alpha$ such
  that $\P(T_\alpha > t) < C e^{-c t}$.
\end{coro}

The last ingredient for \cref{L:TM_tail} is the following lemma.

\begin{lemma}\label{L:dual}
  Assume $\lambda\in\cone(V)^\circ$, then there exists $\alpha>0$ for which
  \[
  \lattice \cap \cone_\alpha(V) \subset \lattice_+.
  \]
\end{lemma}

Thus if a vector can be represented as a combination of vectors of $V$
with sufficiently large coefficients, and can also be represented using
integer coefficients then it can be represented using positive integer
coefficients. Recall that we use superscripts for indices of family types.

\begin{proof}
  For any $\alpha$, and $x\in\R^q$, suppose $x \in \lattice \cap
  \cone_\alpha(v)$. Then $x = \sum n^i\bv^i $ for some integer vector
  $n=(n^1,\dots,n^k)$ and also $x = \sum a^i \bv^i$ for a vector $a$ with
  $\min_i a^i \geq \alpha$. In particular $\sum (a-n)^i \bv^i = 0$. Thus
  we consider the subspace $W \subset \R^k$ of linear relations between the
  elements of $V$, namely $W:=\{b\in\R^k : \sum b^i\bv^i = 0\}$.  Let
  $\lattice^* :=
  W \cap \Z^k$ be the dual lattice of integer vectors in $W$. If $W =
  \{0\}$ then there is a unique way to write each vector $x\in \R^q$ as a
  linear combination of vectors in $V$ and thus $n=a$ and in particular
  $n^i\geq\alpha$. In this case the lemma holds with any $\alpha>0$. We
  assume therefore that $\dim W > 0$.


  We now show that any point in $W$ (and in particular $n-a$) is within
  bounded distance from $\lattice^*$. Since the vectors $\bv^i$ have integer
  coordinates, and since a set of integer vectors is linearly independent
  over the reals if and only if they are linearly independent over the
  rationals (or equivalently, over the integers), $\lattice^*$ contains a
  basis for $W$, which we denote $\ell_1,\dots,\ell_{k-q}$.  We now fix our
  $\alpha$ to be $\alpha=\sum \|\ell_i\|$. Any $w\in W$ we can written in this
  basis as $w = \sum b_i\ell_i$. Let $w' = \sum \lfloor b_i\rfloor \ell_i$,
  then $w'\in\lattice^*$, and $\|w-w'\| \leq \sum\|\ell_i\| = \alpha$.

  Apply the above to $w=a-n\in W$.  Then $w'\in W$, so $x = \sum(n+w')^i
  \bv^i$ is an integer combination of the $\bv^i$s.  Moreover, since
  $\|a-(n+w')\| = \|w-w'\|\leq\alpha$ we find $(n+w')^i \geq a^i - \alpha
  \geq 0$. In particular, $x \in \lattice_+$.
\end{proof}

\begin{proof}[Proof of \cref{L:TM_tail}]
  This follows from \cref{C:T_alpha,L:dual}.
\end{proof}

\begin{prop}\label{P:d1_super}
  If $\lambda\in\cone(V)^\circ$ then there exists a translation invariant
  matching scheme on $\R$ such that $\P^*(X>r) < e^{-c r}$ for some $c>0$.
\end{prop}

Roughly, we search in a greedy manner for intervals containing matchable
sets of points, and partition points in each such interval to valid
families in an arbitrary manner.  The
resulting construction depends on a starting point $0$. From this we
construct a translation invariant matching by considering a stationary
version of a related Markov chain.

\begin{proof}
  Consider the following continuous time Markov chain on $\N^q$.  At rate
  $\lambda_i$ increase the $i$th coordinate.  If the resulting state is in
  $\lattice_+$, jump immediately to $0$.  This corresponds to accumulating
  points along $\R$.  The state gives the number of unmatched points of
  each colour. As soon as it is possible to match {\em all} points yet
  unmatched, we match them and the state reverts to $0$.

  By \cref{L:TM_tail} the time to return to $0$ after leaving it has an
  exponential tail, and thus the Markov chain is positive recurrent, and
  has a stationary distribution. We now use a stationary version of this
  Markov chain in order to construct our matching.

  If the chain moves from
  $v$ to $v+e_i$ at time $t$ then we have a point of colour $i$ at position
  $t$. There is a slight complication since when the chain jumps to state
  $0$ we might not be able to determine from the trajectory of the Markov
  chain what colour of point has just arrived. We could resolve this by
  additional randomness, but instead let us modify the state space to $\N^q
  \times \{1,\dots,q\}$, and use the second coordinate to record the index
  of the last coordinate changed.
  Clearly positive recurrence is maintained, and the
  trajectory of the Markov chain at stationarity determines a Poisson
  processes $\Pi_i$. The times $(T_i)$ at which the chain jumps to
  $0\times\{1,\dots,q\}$ partition $\R$ in a stationary way into intervals
  $(T_i,T_{i+1}]$ so that the points in each interval form a matchable set.
  We can now fix an arbitrary way of matching the points in each of these
  intervals and we are done.
\end{proof}

This concludes the proof of the upper bound for the $d=1$ case. In order to
extend our analysis to $d\geq 2$, we use a dimension reduction trick. A
similar trick has been used in \cite{h-liggett}. The key is to make use of a
suitable random isomorphism between the measure spaces $\R$ and $\R^d$ and
appeal to the $d=1$ case proved above.


\begin{lemma}[\cite{h-liggett}]\label{L:spacefill}
  There exists a random directed graph $H$ with vertex set $\Z^d$
  and only nearest-neighbour edges, with the following properties.
  \begin{enumerate}[label={\textup{(\roman*)}},nosep]
  \item $H$ is almost surely a directed bi-infinite path spanning $\Z^d$.
  \item $H$ is invariant in law under translations of $\Z^d$.
  \item There exists $C=C(d)\in(0,\infty)$ such that for any $x,y\in\Z^d$
    we have almost surely $\|x-y\|^d \leq C d_H(x,y)$, where $d_H$ denotes
    the graph distance along the path $H$.
  \end{enumerate}
\end{lemma}

For a proof, see \cite[Proposition 5]{h-liggett}.  The construction there
is based on taking a random translation of a $d$-dimensional Hilbert curve
(see Figure~\ref{fig:curve}).

\begin{figure}[h]
  \begin{center}
    \includegraphics[width=.6\textwidth]{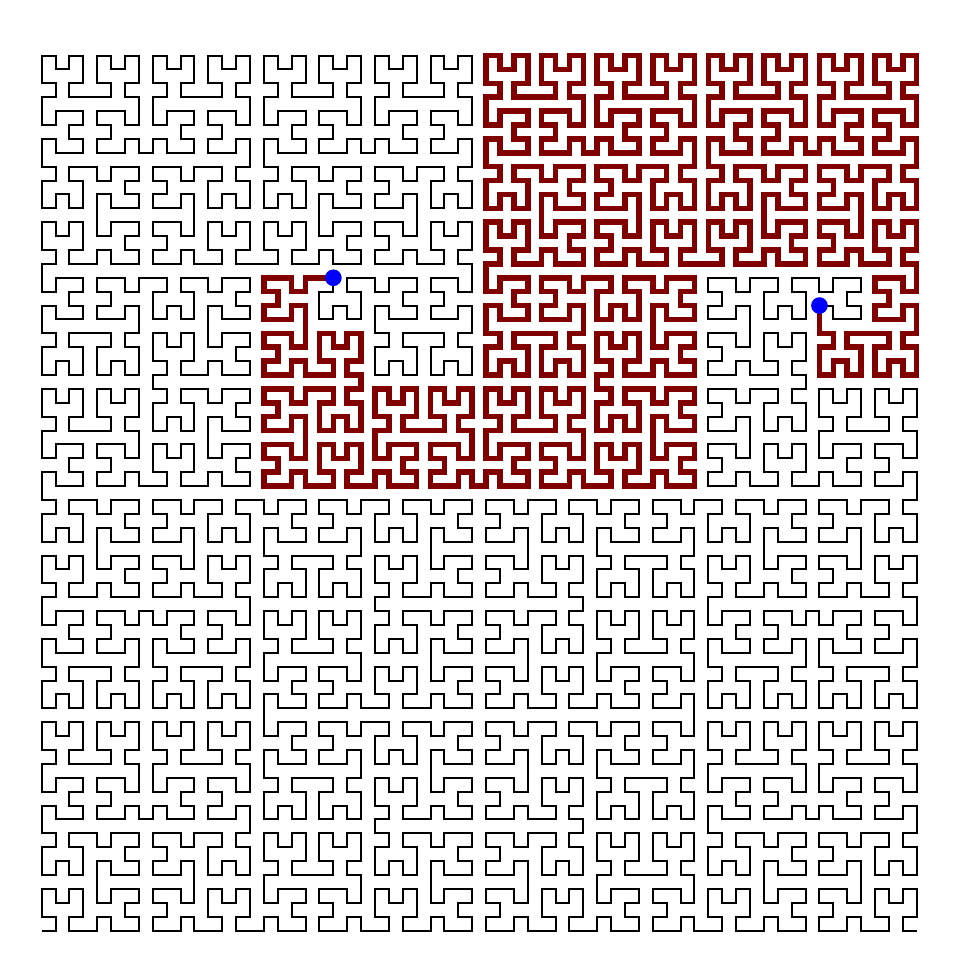}
    \caption{Part of Hilbert's space filling curve. The number of sites
      visited between two given points is at least a constant times the
      square of the distance. A similar construction works in higher
      dimensions.}
    \label{fig:curve}
  \end{center}
\end{figure}

\begin{samepage}
\begin{lemma}\label{L:dim_red}
  Let $\M$ be a translation invariant $V$-matching scheme of independent
  Poisson processes in $\R$, with typical family diameter $X$. Then for any
  $d>1$ there exists a translation-invariant $V$-matching scheme $\M'$ of
  independent Poisson processes in $\R^d$ of the same intensities whose
  typical family diameter $X'$ satisfies $\P(X' > r^d) \leq \P(X > c r -
  C)$ for all $r$, where $c,C$ are constants depending only on $d$.
\end{lemma}
\end{samepage}

\begin{proof}
  Let $\M$ be the $V$-matching of Poisson processes $(\Pi_i)$ in $\R$. It
  suffices to find a $V$-matching $\M'$ on $\R^d$ that is invariant
  in law under translations of $\Z^d$ and satisfies the claimed bound; then
  we obtain a fully translation invariant version by translating $\M'$ by
  an independent uniform element of $[0,1]^d$.

  Let $H$ be as in \cref{L:spacefill} and independent of $\M$. As in
  \cite{h-liggett} we define a bijection $S:\Z\to\Z^d$ by letting
  $S^{-1}(x)$ be the signed graph-distance along the path $H$ from
  $0\in\Z^d$ to $x$ (i.e.\ $\pm d_H(0,x)$, with sign $+$ if the path
  $H$ is directed from $0$ to $x$, and $-$ otherwise).

  Now, given the processes $\Pi_i$ on $\R$ define point processes $\Pi'_i$
  on $\R^d$ as follows. For each point $x\in [\Pi_i]$, let $x'$ be a
  uniform point in the cube $S(\lfloor x \rfloor)+[0,1]^d$, independent of
  all others. Let $\Pi'_i$ be the simple point process whose support is the
  set of resulting points $x'$. Clearly $(\Pi'_i)_{i=1}^q$ are independent
  Poisson process of intensities $\lambda_i$ on $\R^d$. For each family
  $F\in [\M]$ be can define a corresponding family $F'=\{x':x\in F\}$; let
  $\M'$ be resulting $V$-matching. The invariance property in
  \cref{L:spacefill} implies that $\M'$ is invariant in law under
  translations of $\Z^d$.

  If a family $F\in[\M]$ has diameter $r$ in $\R$, then by
  \cref{L:spacefill}(iii) $F'$ has diameter at most $C(r+1)^{1/d} + \surd
  d$ in $\R^d$. The required bound follows.
\end{proof}

\begin{remark*}
  It is possible to avoid the discretization to $\Z^d$ by constructing a
  (continuous) space filling curve $H:\R\to\R^d$ which is an isomorphism of
  measure spaces and satisfies $\|x-y\|^d \leq |H^{-1}(x)-H^{-1}(y)|$ for
  a.e.\ $x,y$. The construction is not very different from that of
  \cref{L:spacefill}. If $H$ is considered only up to translation of its
  parameter, then it can be made translation invariant (this is analogous
  to taking a directed path and not a bijection from $\Z$ to $\Z^d$). Then
  we simply set $\Pi'$ be the push-forward of $\Pi$ and $\M'$ the push-forward
  of $\M$ under the diagonal action of $H$.
\end{remark*}

\pagebreak

\begin{proof}[Proof of \cref{T:main} (iii), case
  $\lambda\in\cone(V)^\circ$]
  The upper bound is a combination of \cref{P:d1_super,L:dim_red}.

  The lower bound is trivial: By our assumptions on $V$ there is at least
  one unit vector that is not in $V$, meaning there is at least one colour
  $i$ for which a single point of colour $i$ is not a legal family. The
  lower bound now follows from the event of a having a point of colour $i$
  in the unit cube with no other points (of any colour) within radius
  $r$. This event has probability $c e^{-c r^d}$
\end{proof}

\section{Factor matchings}\label{s:factor}

Here we prove \cref{T:factor}.  The proof builds on some of the ideas from
the proof of \cref{T:main}, but additional ideas are needed.  Indeed, as we
shall discuss below, the construction giving the upper bound in the
underconstrained case above can be seen as a special case of the construction
used for \cref{T:factor}\ref{factor_exp}.  Note that throughout this
section we have $d=1$.

\subsection{A lower bound}

We start with the lower bound in the case $\lambda\in\cone(V)^\circ$ and
$\lattice\neq\Z^q$: (clause \ref{factor_cauchy}), which is different from
the exponential tail given by \cref{T:main}.

\begin{lemma}\label{L:factor_lower}
  If $\lattice\neq\Z^q$, then any translation invariant matching that is a
  factor of the Poisson processes has $\E^* X = \infty$.
\end{lemma}

\begin{proof}
  Suppose for a contradiction that there is such a factor matching with
  $\E^* X < \infty$.  For any $t\in\R$, let $R(t)=(R(t)_1,\dots,R(t)_q)$ be
  the vector with $R(t)_i$ the number of points in $[\Pi_i]\cap(t,\infty)$
  whose family intersects $(-\infty,t]$.  Note that if there is a point at
  $a$ which contributes to $R(0)$, then $\diam(\M(a)) \geq a$, and so
  (by a standard property of the Palm process, \cite[eq.\ (5)]{HPPS})
  \[
  \E \left[\sum_i R(0)_i\right] \leq \sum \lambda_i \int_0^\infty
  \P^*(X\geq a) da = \E^* X.
  \]
  Thus in any matching with $\E^* X < \infty$, the coordinates of $R(0)$
  are almost surely finite.  Since $R(t)-R(0)$ is finite, it follows that
  a.s.\ $R(t)$ is finite for all $t$.

  Next, let $Q(t) = R(t) + \Pi(0,t]$, and consider how the process $Q(t)$
  evolves as $t$ is increased across a point of some colour $i$.  If the
  point at $t$ is the minimal of its family then $R$ increases by $\bv-e_i$
  for some $\bv\in V$, and $\Pi(0,t]$ increases by $e_i$, so $Q(t)$
  increases by some family type.  If the point at $t$ is not the minimal of
  its family then $Q$ has no jump at $t$.  It follows that $Q(t)$ is in the
  same coset of $\Z^d/\lattice$ for all $t>0$, and hence $R(t) - R(0) +
  \Pi(0,t] \in \lattice$ for all $t$.

  Consider the three cosets $(R(0)+\lattice, R(t)+\lattice,
  \Pi(0,t)+\lattice)$.  We
  claim that as $t\to\infty$ they converge jointly in distribution to a
  triplet $(X,Y,Z)$ of independent cosets, with $Z$ uniform on
  $\Z^q/\lattice$.  This contradicts the identity $Z=X-Y$ above.  To check
  the claim, note that $R(0)$ may be approximated by some function of $\Pi$
  restricted to $[-A,A]$, in the sense that there is a function of the
  restricted process taking values in $\Z^q/\lattice$ that is equal to
  $R(0)+\lattice$ with high probability.
  Similarly $R(t)$ may be approximated by the
  same function applied to $\Pi$ restricted to $[t-A,t+A]$.  Finally,
  $\Pi(0,t)$ is a sum of independent terms $\Pi(0,A)+\Pi(A,t-A)+\Pi(t-A,t)$,
  and the middle term is asymptotically uniform in $\Z^q/\lattice$.
\end{proof}

\subsection{Exponential tail}

When $\lattice=\Z^q$, the last argument does not give any barrier to
existence of a matching with a thinner tail, and indeed such matchings
exist.  We now adapt the construction from \cref{sec.underconstrained} to
construct a factor matching.  To demonstrate that another idea is needed,
consider a particular case of Example 7: points of a single colour, where
families consist of either two
or three points.  If (starting from some point) we wait for a matchable set
and match it, then we get a partition of the points of $\Pi$ into
consecutive pairs.  Given $\Pi$, there are two such matchings.  The
construction above gives a random one of these,  clearly not a factor of
$\Pi$.  Indeed, \cref{T:factor}\ref{factor_cauchy} shows that any factor
matching with exponential tail for $X$ must incorporate both family types.

The new idea is as follows.  Suppose first that each point of $[\Pi]$ is
given an independent fair coin toss.  Modify the construction above so that
a pair is matched if the coin of the right point is heads, while if the
coin is tails then the next point of $[\Pi]$ is added to form a triplet.
It is not hard to show that if this procedure is applied to the points of
$(T,\infty)$ then the resulting matchings converge as $T\to-\infty$, and
the limit is a factor of $\Pi$ together with the coins.  Clearly we cannot
define these coins as a factor of $\Pi$ while keeping them independent of
$\Pi$.  However, in the construction below we assign a coin to each point
of $\Pi$ as a factor of $\Pi$, by looking at the distances to the previous
point of $[\Pi]$, so that the resulting coins are i.i.d.\ and independent
of the colours.

To make this construction of a factor matching precise in the general case,
we first consider
integer indexed processes.  Without loss of generality we may normalize
$\lambda$ to have $\sum \lambda_i = 1$.  Consider a doubly infinite i.i.d.\
sequence of colours $(\xi_i)_{i\in\Z}$ with distribution given by
$\lambda$.  Consider also an independent sequence of i.i.d.\ Bernoulli(1/2)
random variables $(\epsilon_i)_{i\in\Z}$.

We define the population count in an interval $I\subset\Z$ as the vector
$\Pi(I)$ with $\Pi(I)_i = \#\{t\in I : \xi_t = i\}$.  Define the
\textbf{good block} starting at $s\in\Z$ to be the interval $(s,t]$ where
$t>s$ is minimal such that $\epsilon_t=1$ and $\Pi(s,t]\in\lattice_+$, that is
the points in $(s,t]$ are matchable and the extra variable at $t$ is $1$.

\begin{lemma}\label{L:good_tail}
  With the above notations, there are constants $c,C$ so that for any
  $s\in\Z$, if $(s,t]$ is
  the (unique) good block starting at $s$, then we have that $\P(t-s>x) < C
  e^{-cx}$.
\end{lemma}

\begin{proof}
  This is essentially the same as \cref{L:TM_tail}, and an analogous argument
  works.  There are two differences: points are indexed by $\Z$, and we
  require $\epsilon_t=1$.  As in \cref{L:in_cone_c}, with high probability
  $\Pi(s,t]\in\cone_\alpha$ for all large enough $t$, and by \cref{L:dual}
  we find $\Pi(s,t]\in\lattice_+$ for all large $t$.  Since $\epsilon_t=1$
  with probability
  $1/2$ for each $t$, the good block from $s$ has exponential tail.
\end{proof}

\begin{coro}\label{C:finite_good}
  Almost surely there are only finitely many good blocks $(s,t]$ containing
  0.
\end{coro}

\begin{proof}
  This follows from \cref{L:good_tail} and the Borel-Cantelli Lemma.
\end{proof}

We now consider partitions of $\Z$ into good blocks.  Such a partition
arises from a doubly infinite sequence $(t_i)_{i\in\Z}$ such that
$(t_i,t_{i+1}]$ are all good blocks.  (Two such sequences are considered
equivalent if they differ only in a shift of the indices.)
Given such a partition we can
define a matching of the points by taking some arbitrary matching of the
points in each good block.  We say that the sequence $(t_i)$ is a factor of the
sequences $\xi,\epsilon$ if the indicator of the set $\{t_i\}_{i\in\Z}$ is a
factor.  In that case, so is the resulting matching.  The following is a
key step towards proving \cref{T:factor}(ii).

\begin{prop}\label{P:unique}
  Almost surely, there is a unique partition of $\Z$ into good blocks.
\end{prop}

Towards proving \cref{P:unique}, we will define a Markov chain $(Z_n)$ with
state space $\N^q$.
The state $Z_n$ will be a deterministic function of the previous
state $Z_{n-1}$ together with $\xi_n$ and $\epsilon_n$.  Given $Z_{n-1}$,
first increase by
1 the $\xi_n$ coordinate to give $Z'_n:=Z_{n-1}+e_{\xi_n}$.  The next state
$Z_n$ is equal to $Z_n'$
unless $Z_n'$ is a matchable vector and $\epsilon_n=1$, in which case we
instead set $Z_n=0$.  With the given distribution for $\xi$ and $\epsilon$,
this defines a Markov transition matrix, but apriori
there could be multiple sequences $(Z_n)$ consistent with a given sequence
$(\xi_n,\epsilon_n)$.

Subsequently, we shall deduce from \cref{P:unique} that there is in fact a
unique process $(Z_n)$, which moreover is a factor of $(\xi_n,\epsilon_n)$.

\begin{lemma}\label{L:MC_nice}
  The transition matrix on $\N^q$ defined above is irreducible, aperiodic
  and positive recurrent.
\end{lemma}

\begin{proof}
  If $x\leq y$ coordinate-wise then $y$ is reachable from $x$, since we
  could have $\epsilon_i=0$ for as long as needed.  For every state $x$
  there is a matchable state $y\geq x$ coordinate-wise.  By having only the
  last $\epsilon=1$, we see that $0$ is reachable from $x$.  Thus the chain
  is irreducible.  Since $\lattice=\Z^q$, there are matchable states $y$ of
  any large enough $\|y\|_1$, so the possible return times to $0$ have
  greatest common divisor $1$.  Finally, if $Z_0=0$ then the return time to
  $0$ is the $t$ such that $(0,t]$ is a good block.  Since this has an
  exponential tail, the chain is positive recurrent.
\end{proof}

\begin{lemma}\label{L:couple}
  Let $(Z_n)_{n\in\N}$ and $(Y_n)_{n\in\N}$ be two instances of the Markov
  chain with
  different initial condition $Z_0\neq Y_0$ and using the same colours
  $(\xi_n)$ and coins $(\epsilon_n)$.  Then almost surely the chains agree
  eventually.
\end{lemma}

\begin{proof}
  By \cref{L:MC_nice} and the ergodic theorem there is
  some $M$ such that the set $\{n : \|Z_n\|_1\leq M\}$ has density at least
  $2/3$.  The same holds for $Y$, and therefore there are a.s.\ infinitely
  many times when $\|Z_n\|_1,\|Y_n\|_1\leq M$.  We shall show that each
  time this happens, there is some probability of coupling within some
  bounded time.  By the Markov property, the chains almost surely
  couple eventually.

  From any such states $(Z,Y)$, there is some positive probability that the
  next jump
  to $0$ of $Z$ and $Y$ is at the same time, after which $Z$ and $Y$ agree.
  To see this, note that $\lattice = \Z^q$ implies that there is a sequence
  of colours that, when added to $Z$ and $Y$, will make both matchable.
  If subsequent $\xi$'s are such a sequence with all $\epsilon_i=0$ except
  for the last, then $Z$ and $Y$ jump to $0$ together, as desired.
\end{proof}

\begin{proof}[Proof of \cref{P:unique}]
  Suppose for a contradiction that there are multiple distinct partitions
  of $\Z$ into good blocks.  Each such partition gives rise to a copy
  $(Z_n)$ of the Markov chain with $Z_n=0$ precisely at the the ends of the
  blocks of the partition.  By \cref{C:finite_good} there are only finitely
  many different good blocks containing $0$, and so only finitely many
  different values for the Markov chains at time $0$. By \cref{L:couple},
  the associated Markov chains all agree from some time on.  Thus there is
  some minimal $M\in\Z$ such that all partitions give the same value of
  $Z_M$ (and thus the chains also agree for all $n\geq M$).  This $M$ is a
  translation invariant factor of the sequences $\xi,\epsilon$, which is
  impossible.

  To prove existence, note that the sequence of triplets $(\xi_n, \eps_n,
  Z_n)$ is also a positive recurrent Markov chain. Take a stationary doubly
  infinite sequence of triplets, and note that the sequences $(\xi_n)$ and
  $(\eps_n)$ are i.i.d.\ with marginal laws $\lambda$ and Bernoulli(1/2).
  Thus there is a coupling of the sequences $\xi$, $\epsilon$, and $Z$ with
  the given marginals. The set of times when $Z_t=0$ gives the endpoints of
  a partition to good blocks.
\end{proof}

\begin{proof}[Proof of \cref{T:factor}\ref{factor_exp}]
  Assume without loss of generality that $\sum\lambda_i=1$.
  Consider the Palm process of the Poison process $\Pi$, with law $\P^*$.
  First, we construct the process $(\xi_n,\epsilon_n)$ from $\Pi$.  Index
  the points of $\Pi$ by $\Z$ in order, with the point at $0$ having
  index $0$.  Let $\xi_n$ be the colour
  of the $n$th point, and let $\epsilon_n$ be $1$ if the distance from the
  $n$th point to the previous one is at least $\log 2$.  This constructs on
  the probability space of $\Pi$ the i.i.d.\ sequence of colours $\xi$ and the
  independent collection of variables $\epsilon$.

  By \cref{P:unique} there is a unique realization of the Markov chain
  $Z_n$ driven by $\xi$ and $\epsilon$.   As noted above, this gives a
  matching as a factor of the discretized process.  Since points of $\Pi$
  naturally correspond to $\Z$, this also gives a matching as a factor of
  $\Pi$.

  It remains to see that in this matching, $X$ has an exponential tail.  If
  the family of the point at $0$ contains a point greater than $t$, then
  either $[0,t]$ contains less than $t/2$ points, or else the good block of
  $0$ contains at least $t/2$ points.  Both of these events have
  exponentially decaying probabilities in $t$, since $\sum \lambda_i=1$.
\end{proof}

We remark that it is also possible to define a continuous time version of
the Markov chain in this section, and use the continuous process to define
the factor matching (taking into account the time since the last event,
which encodes the $\epsilon_i$).  However, the
discretization makes the process easier to define, and is also useful for
formalizing the constructions in the next section.

\subsection{Factor matching construction when $\lattice\neq\Z^q$}

We now extend the ideas from the previous section to construct a factor
matching for the general underconstrained case.  Define the quotient group
$\Gamma := \Z^q/\lattice$, and note that since $\cone(V)^\circ$ is
non-empty, $\lattice$ has full rank, and so $\Gamma$ is a finite group with
canonical homomorphism from $\Z^q$.

Suppose $\lattice\neq\Z^q$, so that $|\Gamma| > 1$.  The reason that the
construction of the previous section fails is that \cref{P:unique} does not
hold.  Indeed there are precisely $|\Gamma|$ partitions of $\Z$ into good
blocks, and no way to select one as a factor of $\Pi$
(this can be proved similarly to \cref{L:factor_lower}).  We introduce two key
ideas in order to
overcome this difficulty, at the expense of a worse tail for $X$.  First,
we modify the Markov chain so that a version of \cref{L:couple}
holds (and hence also a version of \cref{P:unique}).  This is done by leaving
some points unmatched, chosen independently of their colours.  Second, we
iterate the procedure
to deal with unmatched points.  Thus we
have an infinite series of stages, each dealing with the increasingly
spread out left-overs from the previous stages.  The final product of this
argument is the following.

\begin{prop}\label{P:factor_cauchy}
  If $\lambda\in\cone(V)^\circ$, then there exists a translation invariant
  matching that is a factor of the Poisson processes with $\P^*(X>t) \leq
  C/t$ for some constant $C$.
\end{prop}

\begin{proof}[Proof of \cref{T:factor}\ref{factor_cauchy}]
  The last proposition proves the upper bound, while the lower bound is
  given by \cref{L:factor_lower}.
\end{proof}

\paragraph{Overview of the construction.}
The matching will be constructed in stages.  In each stage we shall
construct a partial matching, which consists of valid families but leaves
some points unmatched.  In each stage the partial matching is constructed
using a variation of the Markov chain that we describe below.  The diameter
of families will have exponential tail.  Unmatched points will be matched
in some later stage.  The probability that a point is matched at stage $s$
will decay exponentially in $s$.  However, points at stage $s$ will
typically be far apart.  It will turn out that the dominant contribution to
$\P^*(X>t)$ will come from $s$ with $2^s \approx t$.

\paragraph{Modified Markov chain.}
Recall the Markov chain $(Z_n)$ from the previous section, taking values in
$\N^q$, with steps defined in terms of colours $\xi_n$ and coins
$\epsilon_n$, which at step $n$ increases the $\xi_n$th coordinate, and
possibly jumps to $0$.

We modify this in two ways to define a new process $Z_n$.  First, we allow
an extra value $\xi_n=\emptyset$, signifying that $n$ is ``unoccupied''.
In that case set $Z_n=Z_{n-1}$.  The set of occupied times can have
correlations, so $Z_n$ is no longer a Markov chain, but is still a time
change of a Markov chain.

Let $Y_n\in\Gamma$ be the coset containing $Z_n$.  Note that $Y_n$ is
determined by $Y_{n-1}$ and $\xi_n$, with no need to know $\epsilon_n$,
since the jumps of $Z_n$ to $0$ do not show in $Y_n$.  Moreover, if we take
two copies of this chain driven by the same sequence $\xi$, started at
$Z_0$ and $Z'_0$, then the coset $Y'_n-Y_n$ does not depand on $n$.  If
$Y'_0\neq Y_0$ then the two Markov chains do not couple.

The second  modification is in terms of a third sequence $\zeta$, in
addition to $\xi$ and $\epsilon$.  We make the following
assumptions about their distribution.
\begin{itemize}
\item[(A1)] $\xi_n$ takes values in $\{\emptyset,1,\dots,q\}$.  The
  set $S = \{n:\xi_n\neq\emptyset\}$ (for \emph{support}) is non-empty, and
  its indicator is an ergodic process.
\item[(A2)] Conditioned on $S$, the $\xi|_S$ are i.i.d.\ with law $\lambda$.
\item[(A3)] Conditioned on $S$, the restrictions $\epsilon|_S$ and
  $\zeta|_S$ take independent uniform values in $\{0,1\}$, and are also
  independent of $\xi$.
\end{itemize}
The $\zeta$ variables indicate some points which \emph{might} be left unmatched
(where $\zeta_n=0$).

Given the sequences $\xi,\epsilon,\zeta$, we now define the transitions of
the process $Z$ and its coset process $Y$ using the following procedure.
This process will be a single stage in an
iterative approach, and is depicted in \cref{F:stage}.
Suppose $Z_{n-1}$ is given.
\begin{itemize}[nosep]
\item If $\xi_n=\emptyset$ then set $Z_n=Z_{n-1}$.
\item If $Y_{n-1}=0$ and $\zeta_n=0$, then also $Z_n=Z_{n-1}$.  In this
  case we say a point at $n$ is \textbf{skipped}.
\item Otherwise, let $Z'_n = Z_n + e_{\xi_n}$ (if we reach this case,
  $\xi_n\neq\emptyset$).
\item If $Z'_n\in\lattice_+$ (i.e.\ is matchable) and $\epsilon_n=1$
  then $Z_n=0$.  Otherwise, $Z_n=Z'_n$.
\end{itemize}

\begin{figure}
  \begin{center}
    \includegraphics[width=.95\textwidth]{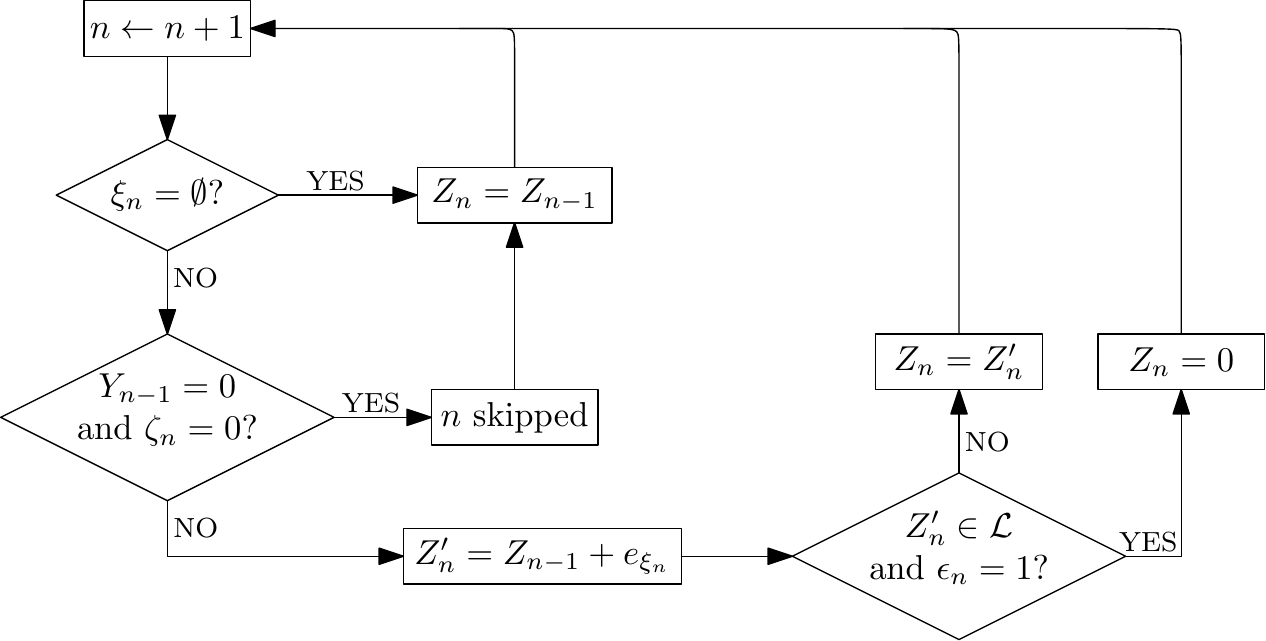}
  \end{center}
  \caption{A stage in the iterative construction of a factor matching.
    This flowchart shows how to compute $(Z_n)_{n\geq a}$ given $Z_a$ and
    the sequences $\xi,\epsilon,zeta$. \cref{P:unique2} implies that there
    is a unique choice of the entire sequence $(Z_n)_{n\in\Z}$ consistent
    with the recursion.}
  \label{F:stage}
\end{figure}

Some observations should be made at this time.  First, since the $\xi$ are
not assumed to be i.i.d., the resulting processes $(Z_n)$, $(Y_n)$ are not
Markov chains.  Instead, these are time changed Markov chains which make a
step at times $n\in S$.  Second, $Y_n$ can be determined from $Y_{n-1}$,
$\xi_n$ and $\zeta_n$, since we always have $Z_n - Z'_n \in \lattice$.
Thus $(Y_n)$ is itself a (time changed) Markov chain.  Finally, let
$\widehat S$ be the set of times at which a point is skipped.  Then the
indicator of $\widehat S$ is also ergodic. (It follows from \cref{P:unique2}
below that $\widehat S$ is a factor of $\xi$ and $\zeta$.)
Moreover, for $n\in\widehat S$, the
procedure above does not observe $\xi_n$.  Consequently, conditioned on
$\widehat S$, the restriction $\xi|_{\widehat S}$ is an i.i.d.\ process with law
$\lambda$.

The first steps in analysis of the process are just as in the case
$\lattice=\Z^q$.  While $Z_n$ and $Y_n$ are not Markov chain, their
transitions at times $n\in S$ are Markovian.  Consider the transition
probabilities of $Z_n$ and $Y_n$ at a time $n\in S$, i.e.\ with $\xi_n$
with law $\lambda$ and independent $\epsilon_n$,$\zeta_n$.

First, we control the return times to $0$ of $Z_n$, in terms of the number
of points in $S$.

\begin{lemma}\label{L:exp_tail}
  Let $T$ be the return time to $0$ for $(Z_n)$, started at $Z_0=0$ and
  $L_t = |S\cap(0,t]|$.  Then for some $c,C>0$ depending only on $V$ and
  $\lambda$ we have $\P(L_T>n) \leq Ce^{-cn}$.
\end{lemma}

\begin{proof}
  This is proved by the exact same argument as \cref{L:good_tail}: With
  high probaility, once $L_n$ is large we have $Z_n\in\cone_\alpha(V)$.  A
  positive fraction of the time it is in $\lattice$ and hence also in
  $\lattice_+$.  Once $Z'_n\in\lattice_+$ and $\epsilon_n=1$ the process
  jumps to $0$.
\end{proof}

As before, if we set $Z_s=0$, and let $t$ be the next time at which
$Z_t=0$, we call $(s,t]$ a good block.  Thus the number of points of $S$ in
the good block starting at $s\in S$ has exponential tail.

\begin{coro}\label{C:exp_tail}
  For some $C,c>0$, the probability that there is some good block $(a,b]$
  containing $0$ and at least $n$ points of $S$ is at most $Ce^{-cn}$.
\end{coro}

\begin{proof}
  There is a unique good block starting at each point of $S$.  For the $n$
  points just to the left of $0$, the probability that the corresponding
  block contains $0$ is at most $C e^{-cn}$ by the previous lemma.  For
  blocks starting further to the left, the probability of containing $0$
  decays exponentially, and the result follows by a union bound.
\end{proof}

We now deduce that the Markov chains are well behaved.

\begin{lemma}\label{L:MC_nice2}
  The transition probabilities of $(Z_n)$ and $(Y_n)$ define irreducible,
  and positive recurrent Markov chains.
\end{lemma}

\begin{proof}
  For $n\in S$, we have $Y_n = Y_{n-1}+e_{\xi_n}$ unless $Y_{n-1}=0$ and
  $\zeta_n=0$, which gives an irreducible Markov chain.  Since it is
  possible for the chain $Z_n$ to reach $0$ from any state, and reach any
  state from $0$, it is also irreducible.  As in \cref{L:MC_nice}, the
  return times to $0$ have exponential tail, hence $Z_n$ is positive
  recurrent.
\end{proof}

\begin{prop}\label{P:unique2}
  Almost surely, there is a unique partition of $\Z$ into good blocks.
  Moreover, there are unique doubly infinite processes
  $(Z_n)_{n\in\Z}$, $(Y_n)_{n\in\Z}$ consistent with the sequences $\xi$,
  $\eps$ and $\zeta$.
\end{prop}

\begin{proof}
  As in the proof of \cref{L:couple}, starting $Z$ and $Z'$ at any two
  initial states and running them using the same sequences $(\xi_n)$,
  $(\epsilon_n)$, $(\zeta_n)$, we have that $Z_n=Z'_n$ from some time
  on.
  The proof of \cref{P:unique} now applies.
\end{proof}

\paragraph{Partial matchings.} For a process $(\xi_n)$ taking values in
$\{\emptyset,1,\dots,q\}$, a \textbf{partial matching} is a translation
equivariant partition of some
subset of $\{n : \xi_n\neq\emptyset\}$ into finite sets so that the values
of $\xi$ in each set are a valid family type, with some points possibly
left unmatched.  This may be defined formally similarly to matchings of
Poisson processes, and we omit a detailed definition.  Given a partial
matching, we let $X$ denote the diameter of the family of $0$ if $0$ is
matched, and set $Z=0$ if $\xi_0=\emptyset$ or if $0$ is unmatched.  Given
a partition of $S=\{n : \xi_n\neq\emptyset\}$ into skipped points and good
blocks, there is a partial matching of the unskipped points with families
contained in good blocks.

Let $\widehat S$ denote the sets of $n$ such that a point at $n$ is skipped.
For sequences $\xi,\epsilon,\zeta$ satisfying assumptions (A1)--(A3), the
process $1_{\widehat S}$ is ergodic, and has intensity $\alpha$ times the
intensity of $\xi$, for some $\alpha<1/2$. (It is not hard to find
$\alpha=(1+|\Gamma|)^{-1}$ by analyzing the Markov chain $Y$.)

\paragraph{Discretization.}
Having analyzed the modified Markov chains, we are now ready to continue
constructing our factor matching.  We first move to a discretized
version of the Poisson process $\Pi$.  We will then introduce partial
matchings and use the discrete process to construct iteratively our factor
matching.  We
conclude the proof by analyzing the tail of $X$ for the resulting matching.

Assume again that $\sum\lambda_i=1$.  We replace the Poisson process by a
discretized version with some additional random variables reserved for
later use.  Label the points of $\Pi$ by the integers $\Z$ in order, with
the maximal point in $(-\infty,0]$ having label $0$.  (Under $\P^*$ this
point is at $0$.)  Let $T_n$ be the location of point $n$, so that $T_0
\le 0 < T_1$ and the gaps $T_n-T_{n-1}$ for $n\neq1$ are
i.i.d.\ $\Exp(1)$ random variables.  Under $\P^*$, the gap $T_1-T_0$ is
no different, though under $\P$ it is biased by its size.  We now create
a discrete process $\hat\Pi$, by rounding each gap up to an integer,
i.e., define the new locations of points by $\hat T_0 = \lceil T_0
\rceil$ and $\hat T_n-\hat T_{n-1} = \lceil T_n-T_{n-1}\rceil$ for
$n\in\Z$.  Let $\hat\Pi = (\hat\Pi_1, \dots, \hat\Pi_q)$, where
$\hat\Pi_i$ is the resulting process of points of colour $i$ supported
within $\Z$.

Note that if $X=\Exp(1)$ then $\lceil X\rceil$ is a geometric random
variable, and therefore the gaps in $\hat\Pi$ are geometric.  Moreover,
$-\lceil T_0 \rceil = \lfloor -T_0 \rfloor \eqd \lceil X\rceil-1$, so that
$\hat T_1 - \hat T_0 \eqd \lceil X\rceil + \lceil X'\rceil -1$ is a size
biased geometric.  It follows that every point is present in $[\hat\Pi]$
independently with the same probability (which happens to be $1-1/e$),
except that $0\in[\hat\Pi]$ almost surely.

Observe that the fractional part of an exponential is independent of its
integer part.  For each $n\in[\hat\Pi]$ let $U_{\hat T_n}$ be the
fractional part of $T_n-T_{n-1}$.  Conditioned on $\hat\Pi$, the
$(U_n)_{n\in[\hat\Pi]}$ are i.i.d..  Thus we have constructed a coloured
Bernoulli process on $\Z$, where each occupied point is also assigned some
independent continuous random variable.  These variables can be used for
making random choices associated with $n$ as a function of the original
process $\Pi$.

\paragraph{Iterative matching.}
We are ready now to describe the complete construction of the factor
matching. The procedure we describe will have an infinite sequence of
stages.  At each stage $s$ there will be sequences $\xi^{(s)}$,
$\epsilon^{(s)}$ and $\zeta^{(s)}$ as above. The Bernoulli variables
$\epsilon^{(s)}_n$ and $\zeta^{(s)}_n$ will be functions of $U_n$.  The
sequences $\xi^{(s)}$ are more delicate.

The set $S_s$ defined below will be equal to the set $\{n : \xi^{(s)}_n
\neq \emptyset\}$.  These
sequences will satisfy assumptions (A1)--(A3), and so \cref{P:unique2}
applies, and there is a unique partition of $\Z$ into good blocks.  Each
block contains a matchable set of points and some skipped points.  The
sequences $\xi^{(s)}$ will be defined
inductively.  If $n$ has been matched in some stage prior to $s$ then
$\xi^{(s)}_n=\emptyset$.  However, we sometimes set $\xi^{(s)}_n=\emptyset$
also for unmatched points, in order to reserve a point for later stages.

For each $n\in[\widehat\Pi]$, from $U_n$ we define a $\geom(1/2)$ variable
$(G_n)_{n\in\widehat\Pi}$, as well as sequences of i.i.d.\ Bernoullis
$(\epsilon_n^{(s)})_{s\in\N}$ and $(\zeta_n^{(s)})_{s\in\N}$.  We now
define the sets $S_s$ inductively as follows.  The set $S_s$ consists of
all $n$ that were skipped at stage $s-1$, as well as all points with
$G_n=s$.  Note that there is no stage $0$, so no points are skipped at
stage $0$.

More precisely, we define $S_s$
inductively: At stage $1$ we have $S_1 = \{n : G_n = 1\}$.  At each stage
we define
\[
\xi^{(s)}_n = \begin{cases}
  \emptyset & n \notin S_s, \\
  i & n \in S_s, n\in \hat\Pi_i.
\end{cases}
\]
We generate a partial matching using $\xi^{(s)}$, $\epsilon^{(s)}$ and
$\zeta^{(s)}$.  We denote by $\widehat S_s \cup \{n
: G_n = s+1\}$ (see \cref{F:iterations}).
Thus $n\in S_s$ if $G_n \leq s$ and $n$ has not been matched in any
previous stage.
 If $n\in S_s$ then we say that $n$ is \textbf{active} at stage $s$.

\begin{figure}
  \begin{center}
    \includegraphics[width=.85\textwidth]{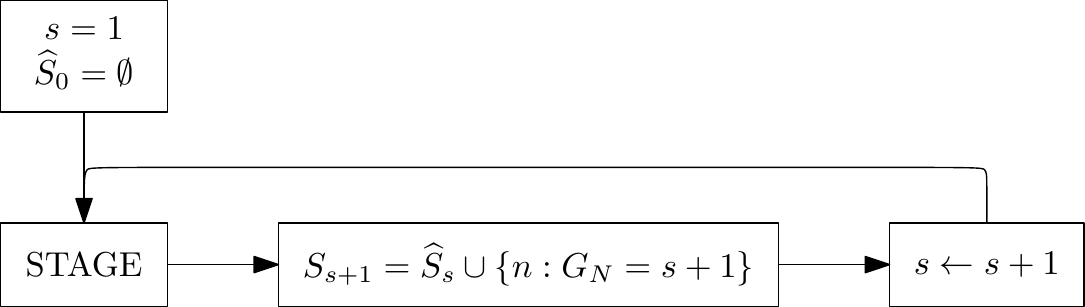}
  \end{center}
  \caption{The iterations in the construction of the factor matching. Here
    "STAGE" refers to defining a partial matching on points in $S_s$, and
    setting $\widehat S_s$ to be the set of skipped points.}
  \label{F:iterations}
\end{figure}

\paragraph{Family size distribution.}
We are finally able to complete our proof of
\cref{T:factor}\ref{factor_cauchy}.

\begin{proof}[Proof of \cref{T:factor}\ref{factor_cauchy}]
  The lower bound is given by \cref{L:factor_lower}.  For the upper bound,
  we analyze the sequential matching procedure described above.  First,
  note that for any $n\in\hat\Pi$ there is a.s.\ some $s\geq G_n$ with
  $\zeta^{(s)}_n=1$, so a point at $n$ is a.s.\ matched at some stage, and
  the iterative procedure indeed yields a factor matching.  It remains to
  estimate the tail of $X$ under $\P^*$.

  Recall that under $\P^*$, there is a point at $0$, and let $\hat X$ be the
  diameter of its family in the discrete process.  Points of $\hat\Pi$ are
  in natural correspondence with points of $\Pi$, and distances in
  $\hat\Pi$ are larger, so deterministically $X\leq\hat X$.  Note that
  $\hat\Pi$ is not ergodic since there is always a point at $0$.  However,
  it is an ergodic process conditioned to have a point at $0$.  Since this
  condition has positive probability for the discrete process, any almost
  sure statement about the ergodic process applies also to $\hat\Pi$, and
  any bound on a probability holds with a constant factor.  We therefore
  consider from now on a process where there is a point at $0$ with the
  same probability as any other $n$.

  The idea for bounding the probability that $\hat X$ is large, is that one
  of several things must
  happen.  Either $0$ is matched at a late stage, or else the good block
  containing $0$ at an early stage $s$ is atypically large.  In the latter
  case, either there are many active points in the good block, or there are
  unusually few.  We show that all of these are unlikely.

  Let us consider the process $1_{S_s}$ of points which are active at stage
  $s$. This is an ergodic process, and our first task is to compute its
  intensity, denoted $\beta_s$.  Let $\gamma$ be the intensity of $\hat\Pi$
  (which happens to be $(1-1/e)$, though the value is not important to us).
  Then $\beta_1 = \gamma/2$, as these are the points of $\hat\Pi$ with
  $G_n=1$.  Each point of $S_s$ is skipped, and so in $\hat S_s$ with
  probability $\alpha$,
  so recursively, $\beta_s = \alpha \beta_{s-1} + 2^{-s}\gamma$, with
  $2^{-s}\gamma$ being the intensity of points of $\hat\Pi$ with $G_n=s$.
  This leads to $\beta_s = \frac{\gamma}{1/2-\alpha} (2^{-s}-\alpha^s) \leq
  C2^{-n}$ for some universal constant $C$. (Recall that $\alpha<1/2$.)

  Fix some $t$.
  We split the event $\hat X>t$ according to the stage at which $0$ is
  matched.  If this stage is $s$, then necessarily $0\in S_s$.
  Let $I\subset\Z$ be the good block in stage $s$ containing $0$.  We have
  \[
  \P^*\big(X \geq t\big) \leq \P^*\big(\hat X \geq t\big)
  \leq \sum_s \P^*\big(0\in S_s, |I| \geq t\big)
  \]

  Now for any $s$ we have
  $\P^*(0\in S_s) = \beta_s/\gamma \leq (C/\gamma) 2^{-s}$ (the factor of
  $1/\gamma$ is since $\P^*$-almost surely $0\in\hat\Pi$).  Thus the
  contribution to the above sum from $s$ with $2^s>t$ is at most
  $\frac{2C}{\gamma t}$.  Let us focus on $\P^*(0\in S_s, |I| \geq t)$ for
  smaller $s$.

  For some $\delta>0$ to be specified below, let
  $B_s = \{|S_s\cap I| > \delta\beta_s t\}$ be the event that $I$ contains
  many points of $S_s$.  By \cref{C:exp_tail},
  $\P(B_s|0\in S_s) \leq C_1 e^{-c_2\delta \beta_s t}$ for some constants.

  Let $D_s$ be the event that at least one of $(0,t/2]$ and $[-t/2,0)$
  contains at most $\delta \beta_s t$ points of $S_s$.  The set $S_s$
  includes all $n$ with $G_n=s$, which is a Bernoulli percolation with
  intensity $2^{-s}\gamma$.  Thus the number of points of $S_s$ in an
  interval $J$ dominates a binomial $\Bin(|J|, 2^{-s}\gamma)$.  We fix
  $\delta$ so that $2^{-s}\gamma t/2 > 2\delta\beta_s t$, i.e. the binomial
  gives in expectation twice as many points as are allowed on the event
  $D_s$.  By a standard large deviation estimate for binomials,
  $\P(D_s|0\in S_s) \leq C_3e^{-c_4\beta_s t}$.

  Observe that $0\in S_s\cap\{|I| \geq t\}$ implies $B_s\cup D_s$, and
  therefore
  \begin{align*}
    \P(0\in S_s\cap\{|I| \geq t\})
    &\leq \P(0\in S_s) \left[ \P(B_s|0\in S_s) + \P(D_s|0\in S_s) \right] \\
    &\leq (C/\gamma) 2^{-s} \left[ C_1 e^{-c_2\delta \beta_s t} +
      C_3e^{-c_4\beta_s t} \right] \\
    &\leq C_5 2^{-s} e^{-c_6 t/2^s}.
  \end{align*}
  It is easy to verify that summing this over $s$ with $2^s\leq t$ gives a
  total of order $1/t$.
\end{proof}

\section{The critical case}
\label{sec.critical}

To complete the proofs of \cref{T:main,T:factor}, we turn to the critical case.
As noted above, the simplest example of this case is matching points of two
Poisson processes of equal densities. In this case the upper and lower
bounds were proved in \cite{HPPS} (in all dimensions). Our proof of the
lower bound follows a similar argument to the one in \cite{HPPS}. To prove
the upper bound, we reduce the general critical case to the two colour
case.  The upper bound of \cref{T:main}(ii), as well as
\cref{T:factor}\ref{factor_crit} both follow from the construction in
\cref{L:critical_upper}.

\subsection{Upper bound: constructing a matching}

\begin{lemma}\label{L:critical_upper}
  With the notations of \cref{T:main}, suppose $\lambda
  \in \partial(\cone(V))$.  Then for some $C$, there exists a
  translation-invariant $V$-matching scheme such that for all $r$
  \[
  \P^*(X>r)\leq \begin{cases} C r^{-d/2} & d\leq 2, \\ e^{-C r^d} & d>2.
  \end{cases}
  \]
\end{lemma}

\begin{proof}
  Our starting point is \cite[Theorem~1]{HPPS} which shows that there
  exists a two-colour matching between equal intensity Poisson process with
  the desired tail bounds. In our notations, this is the case $q=2$, with
  $\lambda_1 = \lambda_2$ and $V=\{(1,1)\}$. We proceed to generalize this
  in several steps.

  The next case we consider is $V=\{(1,1,\dots,1)\}$, i.e.\ there is just
  one family type, consisting of one point of each colour. Since $\lambda
  \in \cone(V)$, all the Poisson processes must have the same intensity,
  which without loss of generality we may assume is $1$. To construct the
  required matching, start with a Poisson process with unit intensity
  $\Pi_1$.  The two-colour matching conditioned on
  $\Pi_1$ gives a law for a second Poisson process together with a matching
  between its points and those of $\Pi_1$. Take $q-1$ independent samples
  from this conditional law.  Together with $\Pi_1$ we now have $q$
  independent Poisson
  processes as well as matchings between $\Pi_1$ and each of the others.
  This gives a natural partition of the points of all $q$ Poisson processes
  into valid families.  Since there are only finitely many colours, up to
  constants this matching scheme has the same tail behaviour as the
  two-colour scheme: $\P^*(X>r) \leq (q-1)\P^*(X'>r/2)$, where $X'$ is the
  distance in the two-colour matching.

  The next step is the case $V=\{\bv\}$, where there is just a single
  family type.  Note that necessarily $\lambda = a \bv$ for some $a>0$.  To
  construct the matching, let $q' = \sum \bv^i$, and start with a
  $q'$-colour matching where all colours have intensity $a$ and all
  families have type $(1,1,\dots,1)$. Such a matching scheme exists by the
  previous paragraph. Next, partition the $q'$ colours into $q$ classes,
  with $\bv^i$ colours in the $i$th class. Let $\Pi_i$ be the sum of the
  Poisson processes of colours in the $i$th class. Clearly taking the
  resulting Poisson processes with the same partition to families yields
  the resulting matching scheme, with the same distribution for $X$.

  Finally we consider the general case. Since $\lambda\in\cone(V)$, we can
  write $\lambda = \sum_{j=1}^k a^j \bv^j$ for some non-negative
  coefficients $a^j$. For each $j$, consider an independent matching scheme
  $\M_j$ as above with a single family type $\bv^j$, of Poisson processes
  $\{\Pi_{i,j}\}_{i\leq q}$ with intensities given by the vector $a^j
  \bv^j$. Let $\M = \sum_j \M_j$ and $\Pi_i = \sum_j \Pi_{i,j}$. Then $\M$
  is a valid matching scheme of Poisson processes with intensities given by
  $\lambda$. Finally, $\P^*(X>r) \leq \max_j \P^*(Z_j>r)$ has the required
  tail.
\end{proof}

\begin{lemma}\label{L:factor_critical}
  If $d=1$ and $\lambda\in\partial\cone(V)$, there is a factor matching
  with $\P^*(X>t) \leq C/\surd{t}$.
\end{lemma}

To prove this, we adopt the proof of \cref{L:critical_upper}, together with
some ideas from the proof of \cref{T:factor}.

\begin{proof}
  First, note first that the matching in \cite{HPPS} in the one dimensional
  two colour case is a factor.  (Recall, this matching recursively matches
  a red point to a blue point immediately to its right and removes the
  pair.)  This has the required tail.  This matching is also valid for a
  discretized process with points on a subset of $\Z$, and has the same
  tail for $X$.

  As above, we can write $\lambda = \sum a^j v^j$.  We would like to split
  the process $\Pi_i$ as a sum of processe $\Pi_{i,k}$ for
  $k\leq \sum_j v^j_i$, with $v^j_i$ of these having intensity $a^j$.  We
  can then group processes of different processes into groups associated
  with the family types.  The group for $v^j$ will include the $v^j_i$ of the
  processes $\Pi_{i,j}$ with intensity $a^j$.  Finally, within each family
  we use the two colour matching to match points of the first process with
  points of the others, giving families of type $v^j$.

  Spliting the processes as above requires additional randomness.  To get a
  factor matching with the same tail behaviour, recall from the proof of
  \cref{T:factor}\ref{factor_exp} the discretized process $\hat\Pi$, where
  each point is also assigned an independent continuous random variable
  $U_n$, independent of $\hat\Pi$.  The $U$ variables can be used to split
  the points into sub-processes, where a point is in $\hat\Pi_{i,j}$ with
  probability proportional to $a^j$.  To these we apply the two colour factor
  matching.  Note that the two colour matching works in the same way for
  processes in $\Z$, and has the same tail behaviour for $X$.
\end{proof}

\subsection{Lower bound}\label{s:critical_lower_bound}

It remains to prove the lower bound in the critical case.  When $d>2$ the
bound is the same as in the underconstrained case, and the proof holds with
no change. Thus we are left with the cases $d=1,2$. The proof combines
ideas from \cite{HPPS} with consequences of $\lambda\in\partial\cone(V)$.

In preparation for proving the lower bounds, we introduce some
notations. If $\lambda\in\partial\cone(V)$, then by the supporting
hyperplane theorem, $\cone(V)$ has a supporting
hyperplane at $\lambda$, i.e.\ there is a non-zero $\eta\in\R^q$ with
\[
\eta\cdot\lambda = 0 \geq \eta\cdot \bv^i \qquad \forall i.
\]
We call $\eta_i$ the charge of a point of type $i$.  It is convenient to
denote by $\Psi_\eta$ the weighted measure $\sum \eta_i \Pi_i$.  The charge
of a set $A$ is $\Psi_\eta(A) = \eta \cdot (\Pi_i(A))_i$.  The exact same
mass transport that showed there is no invariant matching scheme in the
unsatisfiable case, shows that any translation invariant matching scheme
a.s.\ includes only family types with $\eta\cdot\bv=0$.

\begin{lemma}\label{L:lower1}
  If $d=1$ and $\lambda\in\partial\cone(V)$, then every
  translation-invariant $V$-matching scheme satisfies $\E^* \surd{X} =
  \infty$.
\end{lemma}

\begin{proof}
  Consider the graph with vertex set $[\Psi] \subset\R$, with an edge between
  every vertex $x$ to the (a.s.\ unique) leftmost and rightmost members of
  $\M(x)$.  For
  the Palm process, let $X' = \max\{|x|, x\in\M(0)\}$ be the length
  of the longest edge incident on the vertex at $0$. Note that $X' \leq X
  \leq 2X'$, so it suffices to prove $\E^* \sqrt{X'} = \infty$.

  Consider now $Q_L$: the number of points $x\in[\Psi]$ in $[-L,L]$ so that
  $\M(x)$ is not contained in $[-L,L]$. Since each family has total charge
  $0$, the points contributing to $Q_L$ have total charge
  $\Psi_\eta([-L,L])$, where $\Psi_\eta$ was defined above. Each point has a
  bounded charge, and hence $Q_L \geq c |\Psi_\eta([-L,L])|$.  By the
  central limit theorem, $\E Q_L \geq c\surd{L}$.

  On the other hand, a point $x\in[-L,L]$ can only count towards $Q_L$ if
  it is attached to distance at least $L-|x|$. Thus $\E Q_L \leq C
  \int_{-L}^L \P^*(X'>L-|x|) dx$. Combining these bounds we get
  \[
  \int_0^L \P^*(X'>t) dt \geq c\surd{L}.
  \]
  which we may write as
  \[
  \int_0^\infty \frac{\P^*(X'>t)}{\surd{t}} \left[1_{t<L} \sqrt{\frac{t}{L}}
    \right] dt \geq c.
  \]
  However, if $\E^* \sqrt{X'} < \infty$ then $\frac{\P^*(X'>t)}{\surd{t}}$
  is integrable, and by the dominated convergence theorem the last integral
  tends to $0$ as $L\to\infty$.
\end{proof}

We now turn to the case $d=2$. We again construct a graph on
$[\Psi]\subset\R^2$ and
also an embedding of the graph in $\R^2$. For two points $x,y$, let
$\dseg{xy}$ be the straight line segment from $x$ to $y$.  For any point
$x\in[\Psi]$, let $y$ be the leftmost point in $\M(x)$.  Then our graph has
a directed edge from $x$ to $y$, embedded in the plane along $\dseg{xy}$.
(This embedding may have crossing edges.  If $x=y$, there is a self-loop,
which will not play any role below.)

\begin{lemma}\label{L:locally_finite}
  For any translation-invariant matching scheme $\M$ with $\E^* X <
  \infty$, the number of such edges in the graph that intersect any bounded
  set $A$ has finite expectation.
\end{lemma}

This is a simple variant of Lemma~10 of \cite{HPPS}, and the proof there
(in the case of 2 colours) applies in our case as well. We include it for
completeness. Note that this holds in any dimension.

\begin{proof}
  For $u\in\Z^d$, let $Q_u$ be the box $u+[0,1]^d$. For $u,v\in\Z^d$ let
  $f(u,v)$ be the expected number of $x$ in $[\Psi]\cap Q_u$ so that the
  segment $\dseg{xy}$ intersects $Q_v$.

  Since a segment of length $\ell$ intersects at most $d(\ell+1)$ cubes,
  and since the length of $\dseg{xy}$ is at most $\diam(\M(x))$, we get
  $\sum_v f(0,v) \leq d(1+\E^* X) < \infty$. By the mass transport
  principle, this equals $\sum_u f(u,0)$, which is therefore finite, but
  this is just the expected number of segments that intersect $Q_0$.

  The claim for any bounded $A$ follows immediately.
\end{proof}

A central tool for the two dimensional case, is the following construction
of a weighted directed graph from any given matching, and an embedding of
the graph in the plane
(with edge intersections allowed). For two points $x,y$, let $\dseg{xy}$ be
the directed line segment from $x$ to $y$. For any point $x\in[\Pi_i]$, let
$y$ be the leftmost point in $\M(x)$, then we have a directed edge from $x$
to $y$, embedded along $\dseg{xy}$, with weight $\eta_i$ (where $i$ is the
colour of $x$). This includes a self loop at $y$. We think of this as a
flow of charge along the directed edge. Since for any family type $\bv$ that
appear in the matching we have $\eta\cdot\bv=0$, a.s.\ the total weight of
all edges entering any given $x$ is $0$.

\begin{lemma}\label{L:lower2}
  If $d=2$ and $\lambda\in\partial\cone(V)$, then every
  translation-invariant $V$-matching scheme satisfies $\E^* X = \infty$.
\end{lemma}

As in the one dimensional case above, the fundamental idea is that
fluctuations in the empirical distribution of colours in a region is
likely to be unmatchable, and to require families that incorporate many
points outside the region.

\begin{proof}
  Without loss of generality we may assume that the matching scheme $\M$ is
  ergodic with respect to the full group of translations of $\R^2$; if not
  we apply the claimed result to the components in its ergodic
  decomposition.  Therefore suppose
  for a contradiction that $\M$ is an ergodic matching scheme satisfying
  $\E^*X<\infty$.

  Consider a flow along the graph defined above, where the flow along an
  edge $\dseg{xy}$ is the charge of $x$ divided by the family size (i.e.,
  $\eta_i/\#\M(x)$ if $x\in[\Pi_i]$). As in the case of $d=1$, all
  families have total charge $0$, so the total flow into any vertex is $0$
  (we include a flow from $x$ to itself along an edge of $0$ length).

  For an ordered pair $u,v\in\R^2$, we define the random
  variable $K(u,v)$ to be the total flux across the directed line segment
  $\dseg{uv}$ from left to right.  Formally, $K(u,v)$ may be defined as the
  sum of $\eta_i$ over all pairs $x,y$ with $x\in[\Pi_i]$, $y$ the leftmost
  point of $\M(x)$, and such that $(u,y,v,x)$ form the vertices of a convex
  quadrangle in counterclockwise order.  Also define $F(u,v) :=
  K(u,v)-K(v,u)$, the net flux across the segment.  We restrict this
  definition to points $u,v$ that are not themselves on any edge $x,y$ of
  the graph, and such that there is no point on the segment $u,v$.

  By \cref{L:locally_finite}, $F(u,v)$ is well defined for all such $u,v$,
  and has finite expectation.  Moreover, if $u,v,w$ are co-linear, then
  $F(u,v)+F(v,w) = F(u,w)$.  A crucial observation is that if
  $S = (u_0,u_1,\dots,u_k=u_0)$ is a
  simple, closed, positively oriented polygon, then (for any flow on a
  graph in the plane) $\sum_i F(u_i,u_{i+1})$
  is the total flow from all vertices of the graph inside the polygon minus
  the total
  flow into those vertices. (An edge may cross the polygon without
  terminating in it, in which case its contribution cancels in the sum.)
  Since in our case, the flow into each vertex is $0$, we find
  \[
  \sum_i F(u_i,u_{i+1}) = \sum \eta_i \Pi_i(S).
  \]
  By a slight abuse of notation we denote the latter sum by $\Psi_\eta(S)$.
  Similarly, if $S$ is negatively oriented the sum is $-\Psi_\eta(S)$.

  Fix vectors $x,u\in \R^2$ with $u\neq 0$. Using the ergodic
  theorem we deduce
  \begin{equation}\label{ergodic}
    \frac{F(x,x+nu)}{n} \xrightarrow{\text{a.s.\ and }L^1} \Phi(x,u)
    \quad\text{ as }n\to\infty,
  \end{equation}
  for some random variable $\Phi(x,u)$ with finite mean. We will show next
  that a.s.\ $\Phi(x,u)$ is constant in $x$, and deduce that it is
  deterministic.

  For linearly independent vectors $a,b\in\R^2$, let $S=S_{a,b}$ be the
  interior of the parallelogram with vertices $0,a,b,a+b$. Then a.s.\
  \begin{equation}\label{parallelogram}
    \pm\Psi_\eta(S_{a,b}) = F(0,a)-F(b,a+b)+F(a,a+b)-F(0,b),
  \end{equation}
  where the sign depends on the orientation of $0,a,a+b,b$ around the
  parallelogram.

  Applying this to $(a,b)=(nu,x)$ with $x$ and $u$ as above, we obtain
  \[
  \frac{\pm\Psi_\eta(S_{nu,x})}{n} =
  \frac{F(0,nu)}{n}-\frac{F(x,x+nu)}{n}+\frac{F(nu,x+nu)}{n}-\frac{F(0,x)}{n}.
  \]
  As $n\to\infty$, the left side converges a.s.\ to $0$ by the strong law
  of large numbers, while the last term converges a.s.\ to $0$ because
  $F(0,x)$ is a.s.\ finite.  An easy application of Borel-Cantelli shows
  that, since for each $n$ we have $F(nu,x+nu)\stackrel{d}{=} F(0,x)$, and
  the latter has finite mean, the third term on the right converges a.s.\
  to $0$.  Thus, using \eqref{ergodic},
  \[
  0 = \Phi(0,u) - \Phi(x,u) \quad \text{a.s.}
  \]
  Thus $\Phi(0,u)$ is a
  translation-invariant function of $\M$, and the ergodicity
  assumption implies that it is an a.s.\ constant, which we denote $\phi(u)$.
  Furthermore, since $F(x,x+nv)$ has the same law as $F(0,nv)$,
  it now follows from \eqref{ergodic} that
  \begin{equation}\label{fixed}
    \frac{F(x_n,x_n +nv)}{n} \xrightarrow[n\to\infty]{L^1} \phi(v)
  \end{equation}
  for any $u\in\R^2$ and any deterministic sequence $x_n\in\R^2$. In
  particular for $x_n=nu$,
  \[
  \frac{F(nu,nu+nv)}{n} \xrightarrow[n\to\infty]{L^1} \phi(v).
  \]
  Now let $u=(1,0)$, $v=(0,1)$, and consider the square $S_{nu,nv}$. By
  \eqref{fixed}
  \[
  \frac{\Psi_\eta(S_{nu,nv})}{n} \xrightarrow[n\to\infty]{L^1}
  \phi(u)+\phi(v)-\phi(u)-\phi(v)=0.
  \]
  On the other hand, by the central limit theorem,
  $\frac{\Psi_\eta(S_{nu,nv})}{n}$ converges in distribution to
  $N(0,\sigma^2)$ for some $\sigma>0$, a contradiction.
\end{proof}

\begin{proof}[Proof of \cref{T:main}(ii), lower bound]
  In the cases $d=1$ and $d=2$, this is precisely \cref{L:lower1,L:lower2}.
\end{proof}

\section{Infinitely many types}
\label{sec.infinite}

The proof of \cref{T:infinite} is mostly the same as in the case of
finitely many colours. We only describe in detail the parts of the proof that
differ.

\begin{proof}[Proof of \cref{T:infinite}]
  If $\lambda$ is outside the closure of the cone then the mass transport
  argument from \cref{sec.unsatisfiable} holds with no change.

  With infinitely many family types, it is possible that $\cone(V)$ is not
  closed.  For example, with family types $(n,n+1)$ for any $n\ge0$, the cone
  is $\{0\leq x < y\}$.  Thus it is possible that
  $\lambda\notin\cone(V)$ but is in the boundary of the cone.  In that
  case, as in \cref{sec.critical} we can choose some $\eta$ with
  \[
  \eta\cdot\lambda = 0 \geq\eta\cdot \bv^i.
  \]
  for all $\bv^i\in V$.  The same mass transport argument now shows that no
  matching
  uses any family $v\in V$ with $\eta\cdot v\neq 0$. Therefore if there is
  a matching scheme, there is one using only the subset $V_1$ of family
  types orthogonal to $\eta$.  In the example above, $V_1=\emptyset$, so
  no matching scheme exists.

  Since $\lambda\notin\cone(V)$, it is also not in $\cone(V_1)$. If
  $\lambda$ is also not in the closure of $\cone(V_1)$ then we are back in
  case (i), and there is no invariant matching scheme. Otherwise, we can
  repeat this procedure with a new vector $\eta_1$, giving a set $V_2$ and
  so on.

  More precisely, if $V_i$ and $\lambda$ are contained in some subspace of
  $\R^q$ of dimension at most $q-i$, and $\lambda\notin\cone(V_i)$, there
  is a non trivial linear functional on the subspace, given by some
  $\eta_i$, that separates $\lambda$ from $\cone(V_i)$.  Restricting to
  family types in the kernel of that operator gives a set $V_{i+1}$ that is
  contained in a subspace of dimension at most $q-(i+1)$.  Thus after at
  most $q$ iterations we find that $V_i$ is empty, and trivially there is
  no invariant matching scheme.

  \medskip

  If $\lambda\in\cone(V)$, then by Carath\'eodory's theorem \cite[Theorem
  17.1]{Rock} $\lambda$ is a linear combination of finitely many of the
  $\bv^i$s (at most $q$, specifically).  Hence $\lambda$ is also in the
  cone of a finite subset $V'\subset V$, and matchings could be constructed
  using only family types in $V'$.

  If $\lambda$ is in the boundary and in the cone then it is also in the
  boundary of the cone spanned by the finite set $V'$. The
  constructions for the critical case of \cref{T:main} apply and we get the
  same upper bounds.

  If $\lambda$ is in the interior of the cone then it is also in the
  interior of the cone spanned by some finite subset $V'\subset V$.  This is
  since for any denumerable $V$ we have $\cone(V) = \bigcup_k
  \cone(\{\bv^1,\dots,\bv^k\})$. Thus we can apply \cref{T:main} to get a
  $V'$-Matching scheme with the claimed upper bound for the typical
  distance by restricting ourselves to families in that finite set.

  The proofs of the lower bounds in the critical and underconstrained cases
  continue to hold verbatim.
\end{proof}

\section{Multicoloured pair matchings}
\label{s:pairs}

In this section we give the proof of \cref{P:pairs}, which relates
existence of deficient and critical sets to the location of $\lambda$
w.r.t.\ $\cone(V)$ (and via \cref{T:main} to the possible tail behaviours
of matchings).

As noted, the claim is essentially a result on existence of fractional
matchings in graphs.  We did not find a reference which also addresses the
issue of critical sets.  Since the proof is short we include it here in its
entirety.

\begin{proof}[Proof of \cref{P:pairs}]
  We apply the Max Flow -- Min Cut Theorem to a network constructed from
  $V$ as follows.  For each colour $i$ there are two vertices denoted $s_i$
  and $t_i$.  There is an edge $(s_i,t_j)$ if and only if $i\sim j$, and
  these edges have infinite capacity.  There is an additional source vertex
  $\sigma$ connected to each $s_i$ by an edge with capacity $\lambda_i$, and
  a target vertex $\tau$ connected to each $t_i$ by an edge with capacity
  $\lambda_i$ (see \cref{fig:flow}).

  \begin{figure}
    \centering
    \includegraphics[width=.8\textwidth]{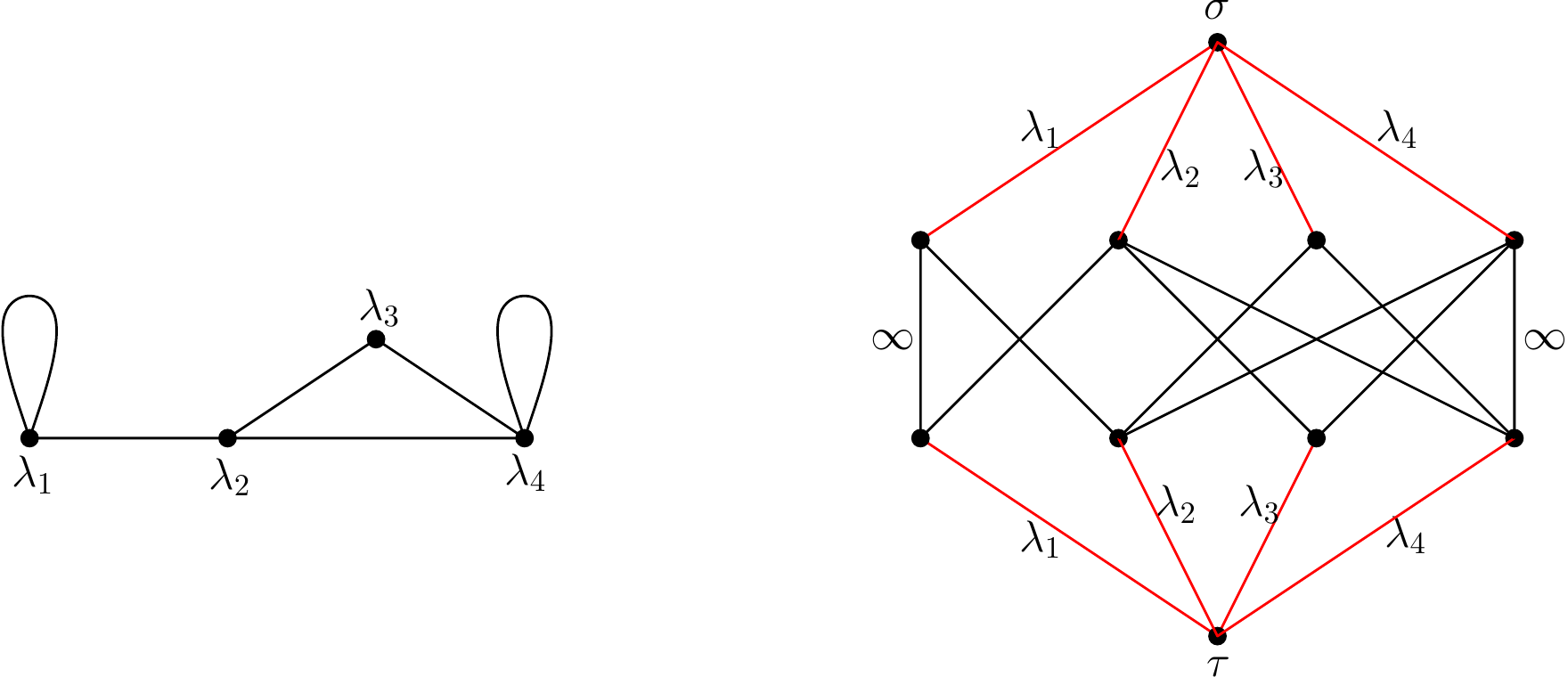}
    \caption{Left: A relation encoded as a graph, with intensities
      $\lambda_1,\dots,\lambda_4$ for the four colours.  Right: The
      corresponding network with edge capacities.}
    \label{fig:flow}
  \end{figure}

  Consider now the maximal flow from $\sigma$ to $\tau$ which is consistent
  with the given edge capacities, and let $f_{ij}$
  be the flow through the edge $(s_i,t_j)$.  By the Max Flow -- Min Cut
  Theorem, the total flow $\sum_{i,j} f_{ij}$ equals the minimal capacity of
  a cutset of edges separating $\sigma$ from $\tau$.  Such a minimal cutset
  cannot contain any edge $(s_i,t_j)$, as these have infinite capacity.  If
  a cutset contains the edges
  $\{(\sigma,s_i)\ :\ i\not\in S\}$, then to be a cutset it must contain
  all edges $\{(t_j,\tau)\ :\ j\in N(S)\}$.  Such a cutset has capacity
  \[
  \lambda(S^c) + \lambda(N(S)) = \|\lambda\|_1 - \lambda(S) +
  \lambda(N(S)),
  \]
  This is strictly less than $\|\lambda\|_1$ if and only if
  $\lambda(S)>\lambda(N(S))$, i.e.\ if $S$ is deficient.  Since taking
  $S=\emptyset$ gives a cutset of capacity $\|\lambda\|_1$, the maximal
  flow is $\|\lambda\|_1$ if and only if there is no deficient set.

  We now argue that a flow of $\|\lambda\|_1$ exists if and only if
  $\lambda\in\cone(V)$.  Indeed, if $\lambda\in\cone(V)$, then we have
  $\lambda = \sum a_{ij}(e_i+e_j)$ with some $a_{ij}=0$ when $i\not\sim j$.
  Take a flow of $a_{ij}$ on each of the edges $(s_i,t_j)$ and $(s_j,t_i)$,
  with the convention that if $i=j$ the flow on $(s_i,t_i)$ is $2a_ii$.
  Take a flow at maximal capacity $\lambda_i$ on the edges $(\sigma,s_i)$
  and $(t_i,\tau)$.  This flow has the required total flow (and conserves
  mass at all vertices).  Conversely, if there is a flow $f$ of size
  $\|\lambda\|_1$, let $a_{ij} = \frac12 (f_{ij} + f_{ji})$ and observe
  that $\sum_{ij} a_{ij} (e_i+e_j) = 2\lambda$, and so
  $\lambda\in\cone(V)$.

  Thus $\lambda$ is outside the cone if and only if there is some
  deficient set $S$.  If every set is excessive then the same holds
  for any sufficiently small perturbation of $\lambda$, and so
  $\lambda\in\cone(V)^\circ$.  If there is some critical set, then since
  $S\not= N(S)$ there is some $x\in S \setminus N(S)$.  Increasing
  $\lambda_x$ by any amount makes the set deficient.  Thus if there is a
  critical set but no deficient set, $\lambda\in\partial\cone(V)$.
\end{proof}

\section{Open questions}\label{s:questions}

\paragraph{Factor matchings in higher dimensions.}
What is the optimal tail behavior of a multicolour matching that is a
factor (i.e.\ a deterministic, translation equivariant function) of the
Poisson processes?  \cref{T:factor} gives some information in the case
$d=1$.  For $d\geq 2$, the condition $\lattice\neq\Z^q$ is no longer a
clear obstacle.  Do there exist matchings with the same tail behaviour as
in \cref{T:main} even if $\lattice\neq\Z^q$?

\paragraph{Stable matchings.}
When the allowed families all have size two, a matching is called
\emph{stable} if there do not exist two points that are closer to each
other than to their respective partners, but that could form a legal
family.  Stable matchings in the one-colour and two-colour cases are
investigated in \cite{HPPS}.  For general multicolour matching in pairs,
when does a perfect stable matching exist?  When a stable matching exists,
what can be said about $X$?  See \cite{HMP} for some progress in certain
cases.  When families may have more than two elements, there are many
possible non-equivalent extensions of the notion of stability, and the
questions of existence and properties are also of interest.

\paragraph{Non-crossing matchings.}
A matching into pairs is called {\it non-crossing} if the line segments
joining the points of each pair are pairwise disjoint.  For processes in
$\R^2$, the question of existence of a non-intersecting invariant matching
in two dimensions is open even for the case of two colours of equal
intensities.  See \cite{H-geom}.  Again, there are several ways to
generalize this notion to other family types. For instance, one can ask
that there is some choice of line segments connecting the points of each
family so that the sets of line segments do not intersect each other.
Alternatively, one could ask that the convex hulls of the families are
disjoint. In the latter sense the question is not trivial in higher
dimensions $d$, provided $d$ is at most twice the maximum family size.

\paragraph{Minimal matchings.}
Still in the setting of matching in pairs, a matching is called
\emph{minimal} if any other valid matching resulting by re-matching some
finite subset of points has a larger total length.  This notion too can be
extended in different ways to matchings with families of other sizes.
Under what conditions does a minimal matching exist?  This is open even in
the case of two colour matching.

\paragraph{Refined tail behavior.}
The lower and upper bounds on the tail of $X$ are generally close, but a
gap still exists.  For example, in the critical two dimensional case we
know that $\E^* X = \infty$ and that ther is a matching with
$\P^*(X>r) \leq C/r$.  Could there be a matching with
$\P^*(X>r) < C/(r\log r)$?  In the underconstrained case there are lower
and upper bounds $e^{-C r^d} \leq \P(X>r) \leq e^{-c r^d}$.  Can these
bounds be replaced by $\P^*(X>r) = e^{-a r^d + o(r^d)}$ with the same
constant $a$ for both sides?

\bibliography{matchings}

\begin{thebibliography}{10}

\bibitem{BLPS}
Itai Benjamini, Russell Lyons, Yuval Peres, and Oded Schramm.
\newblock Uniform spanning forests.
\newblock {\em Ann. Probab.}, 29(1):1--65, 2001.

\bibitem{BS}
Itai Benjamini and Oded Schramm.
\newblock Percolation in the hyperbolic plane [mr1815220].
\newblock {\em J. Amer. Math. Soc.}, 14(2):487--507 (electronic), 2001.

\bibitem{H-geom}
Alexander~E. Holroyd.
\newblock Geometric properties of {P}oisson matchings.
\newblock {\em Probability Theory and Related Fields}, 150(3):511--527, 2011.

\bibitem{h-liggett}
Alexander~E. Holroyd and Tom~M. Liggett.
\newblock How to find an extra head: optimal random shifts of {B}ernoulli and
  {P}oisson random fields.
\newblock {\em Ann. Probab.}, 29(4):1405-1425, 2001.

\bibitem{HMP}
Alexander~E. Holroyd, James~B. Martin, and Yuval Peres.
\newblock Asymmetric stable matchings in high dimensions.
\newblock In preparation.

\bibitem{HPPS}
Alexander~E. Holroyd, Robin Pemantle, Yuval Peres, and Oded Schramm.
\newblock Poisson matching.
\newblock {\em Ann. Inst. Henri Poincare Probab. Stat.}, 2009.

\bibitem{kall}
O.~Kallenberg.
\newblock {\em Foundations of Modern Probability}.
\newblock Springer, 2nd edition edition, 2002.

\bibitem{lyons}
Russell Lyons and Yuval Peres.
\newblock Probability on trees and networks. book in preparation, 2015.

\bibitem{Rock}
R.~T. Rockafellar.
\newblock {\em Convex analysis}.
\newblock Number~28 in Princeton Mathematical Series. Princeton University
  Press, 1970.

\bibitem{fgt}
Edward~R. Scheinerman and Daniel~H. Ullman.
\newblock {\em Fractional Graph Theory}.
\newblock Wiley and Sons, 2008.

\bibitem{Timar}
Adam Timar.
\newblock Invariant matchings of exponential tail on coin flips in $z^d$.
\newblock arXiv:0909.1090, 2009.

\bibitem{tutte}
William~T. Tutte.
\newblock The factorization of linear graphs.
\newblock {\em J. London Math. Soc.}, 22:107--111, 1947.

\end{thebibliography}

\bigskip

\begin{enumerate}[font=\sc, leftmargin=55mm, labelsep=5mm]

\item[Gideon Amir]
  Bar Ilan University, Ramat Gan, Israel \\
  \verb+gidi.amir@gmail.com+

\item[Omer Angel]
  University of British Columbia, Vancouver, Canada \\
  \verb+angel@math.ubc.ca+

\item[Alexander E. Holroyd]
  Microsoft Research, Redmond, USA \\
  \verb+holroyd@microsoft.com+

\end{enumerate}

\end{document}